\documentclass[10pt]{article}

\DeclareMathAlphabet{\mathpzc}{OT1}{pzc}{m}{it}

\usepackage{amsmath}
\usepackage{amssymb}
\usepackage{amsfonts}
\usepackage{amsthm}
\usepackage{mathrsfs}
\usepackage{ifsym}
\usepackage{fontenc}
\usepackage{textcomp}
\usepackage{tipa}
\usepackage{enumerate}
\usepackage[pdftex]{color, graphicx}
\usepackage{graphics,epsfig,color}
 \usepackage{hyperref}
\usepackage[letterpaper]{geometry}
\begin{document}

\title{{\bf Arithmetic properties of $3$-cycles of quadratic maps over $\mathbb{Q}$}}         
\author{Patrick Morton and Serban Raianu}        
\date{January 16, 2022}          
\maketitle

\begin{abstract}  It is shown that $c=-29/16$ is the unique rational number of smallest denominator, and the unique rational number of smallest numerator, for which the map $f_c(x) = x^2+c$ has a rational periodic point of period $3$.  Several arithmetic conditions on the set of all such rational numbers $c$ and the rational orbits of $f_c(x)$ are proved.  A graph on the numerators of the rational $3$-periodic points of maps $f_c$ is considered which reflects connections between solutions of norm equations from the cubic field of discriminant $-23$.
\end{abstract}

\section{Introduction.}

In this paper we will take another look at arithmetic properties of the quadratic map $f_c(x) = x^2 + c$ on $\mathbb{Q}$, focusing on its periodic points of order $3$.  These periodic points have been parametrized in \cite{m} and \cite{rw} and display some fascinating properties.  Periodic points of order $3$ are interesting for several reasons.  For one thing, Sharkovskii's Theorem (see \cite{de}, \cite{ro}, \cite{ly} and the references in \cite{bb}) says the following over $\mathbb{R}$: if a continuous map on an interval has a periodic point of period $3$, then it has real periodic points of all periods.  Secondly, rational periodic points of period three occur for infinitely many quadratic maps which also have $3$ preperiodic points in $\mathbb{Q}$ and therefore at least $6$ periodic and preperiodic points in $\mathbb{Q}$ altogether \cite{m}, \cite{si}.  It has been conjectured by Poonen \cite{p} that a quadratic map over $\mathbb{Q}$ can have no more than $8$ rational periodic or preperiodic points altogether.  In particular, one would like to show that any quadratic map with a rational $3$-cycle has no other periodic points in $\mathbb{Q}$; this has been shown to be true for fixed points and points of period $2$ by Poonen \cite{p}.  In other words, if a map $f_c(x)$ with $c \in \mathbb{Q}$ has a periodic point of period $3$ in $\mathbb{Q}$, then it does not have either fixed points or points of period $2$ in $\mathbb{Q}$.  The corresponding assertion is automatically true for periods $4$ and $5$, since it has been shown that no quadratic map over $\mathbb{Q}$ can have rational points with these periods.  (See \cite{m1}, \cite{fps}.)  In \cite{fps} it is conjectured that the map $f_c(x)$ (for $c \in \mathbb{Q}$) has no rational $n$-cycles for $n \ge 6$, but this is implied by Poonen's conjecture.  This follows from the fact that $f(x) = x^2 + c = a$ generally (for $a \neq c$) has two rational solutions whenever it has one, so any map with a rational $n$-cycle also has at least $n-1$ rational preperiodic points.  (In this paper we use the term {\it rational $n$-cycle} to refer to the orbits of rational periodic points of minimal period $n$.)  \medskip

In the center of this discussion sits the map $f_{-29/16}(x) = x^2 -\frac{29}{16}$, which is the only map with $3$ rational periodic points and $5$ rational preperiodic points (see \cite{p}).  We will characterize this map by showing in Theorems \ref{thm:1} and \ref{thm:4} that $c = -\frac{29}{16}$ is the rational number of smallest height for which $f_c(x)$ has a rational $3$-cycle.  (The height of a rational number is the maximum of the absolute values of its numerator and denominator.)  If the aforementioned conjecture is true, then the maximum number of rational periodic and preperiodic points would occur for the map $f_c(x)$, for which $c$ has smallest height (among all rational numbers for which $f_c(x)$ has a rational $3$-cycle). \medskip

\begin{figure}[h!]
        \centering
        \includegraphics[scale=.7]{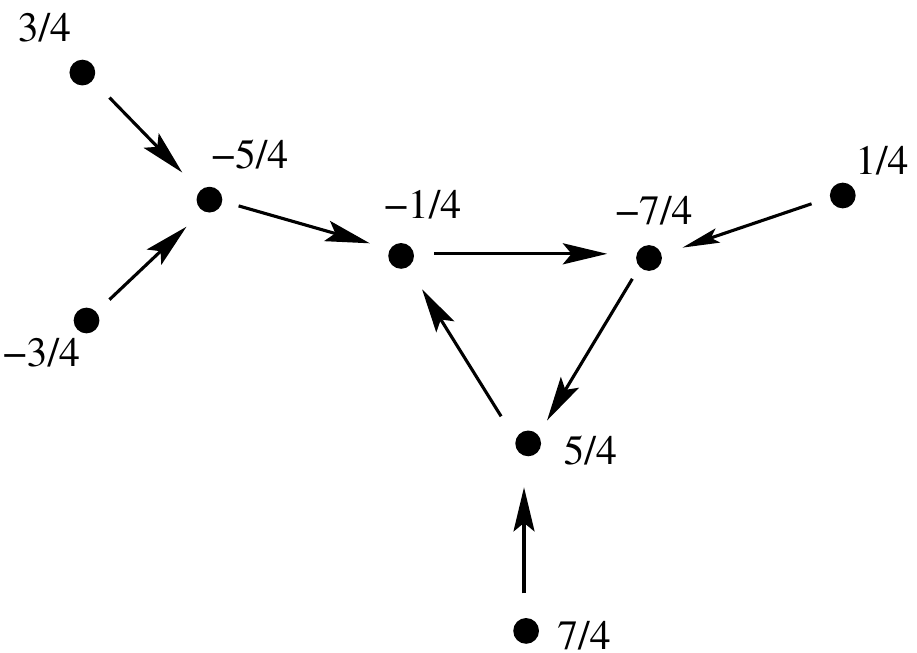}
        \caption{Rational preperiodic and periodic points for $x \mapsto x^2 -\frac{29}{16}$.
\label{fig:1}}
\end{figure}

The rational parametrization of the $c$-values for which $f_c(x)$ has a point of period three, which is valid over any field whose characteristic is not $2$, has especially interesting arithmetic properties over $\mathbb{Q}$.  For one thing, we can characterize the prime factors of the numerator of such a number $c$, as well as the prime factors of the numerators of the rational numbers $\{x_1, x_2, x_3\}$ which make up its $3$-cycle.  This is because the numerators in the parametrization turn out to be expressible as norms: the numerator of $c$ is a norm from $\mathbb{Q}(\zeta_7)$, where $\zeta_7$ is a primitive $7$-th root of unity (Theorem \ref{thm:2}); and the numerators of the $x_i$ are norms from the cubic field $\mathbb{Q}(\gamma)$, where $\gamma^3-\gamma-1 = 0$ (Theorem \ref{thm:5}).  (This $\gamma$ generates the Hilbert class field over the field $\mathbb{Q}(\sqrt{-23})$.)  This allows us to prove a minimum property for $c = -\frac{29}{16}$; namely, that it is the rational number of smallest numerator (in absolute value), for which $f_c(x)$ has a rational $3$-cycle (Theorem \ref{thm:4}).  It is also one of only two values of $c$, for which some $x_i$ in the $3$-cycle has a numerator which is $\pm 1$ (Theorem \ref{thm:7}(a)). \medskip

The paper is laid out as follows. In Section 2 we prove in an elementary way that $c=-\frac{29}{16}$ is the unique rational number of smallest denominator, for which $f_c(x)$ has a rational $3$-cycle, as well as the only such rational number whose denominator is a power of $2$.  In Section 3 we use basic algebraic number theory to show that $c = -\frac{29}{16}$ has the minimum property with respect to its numerator.  This requires us to use several well-known facts about the cyclotomic field $\mathbb{Q}(\zeta_7)$.  We also exhibit the numerators of the $3$-periodic points $x_i$ as norms and determine when one of them is $\pm 1$, $\pm 5$ or $\pm 7$ (the smallest possible values).  This requires us to solve several Thue equations, which boils down to determining when the $n$-th term of a certain linear recurring sequence is zero.  \medskip

In Section 4, we define an infinite graph $\Gamma$ on the absolute values of possible numerators of $3$-periodic points of $f_c(x)$, as $f_c$ varies over all maps with rational $3$-cycles: in this graph two positive integer nodes $a, b$ are connected by an edge if $\pm a, \pm b$ occur as distinct numerators of elements of the rational $3$-cycle of $f_c(x)$, for some $c$.  See Figures \ref{fig:2}-\ref{fig:4}.  This graph $\Gamma$ has some interesting and mysterious properties which reflect the arithmetic in the field $K = \mathbb{Q}(\gamma)$.  We state a number of conjectures concerning this graph.  For example, its connected components all appear to be finite subgraphs, and the number of triangles in $\Gamma$ which meet at a given vertex $a$ seems to equal the number of {\it allowable} solutions of the Thue equation
$$m^3+2m^2n+mn^2+n^3 = a,$$
meaning solutions $(m,n) \in \mathbb{Z}^2$ for which $mn(m+n) \neq 0$ and $\gcd(m,n) = 1$.  In Theorem \ref{thm:10} we show that the last conjecture is true, if every triangle in $\Gamma$ is a {\it $c$-triangle}, meaning a triangle which connects the numerators of the rational numbers in a $3$-cycle of $f_c$, for some $c$.  In Theorem \ref{thm:9} we prove that the three numerators in a $3$-cycle determine the cycle and the map $f_c$.  We conjecture that two of the numerators in such a $3$-cycle are enough to determine the cycle, which Theorem \ref{thm:11} shows is true for all but finitely many pairs of numerators.  Then in Theorem \ref{thm:12} we show that there are infinitely many vertices in $\Gamma$ at which three triangles meet.  Moreover, the vertices we exhibit in this theorem are norms from the Hilbert class field $\Sigma$ of $\mathbb{Q}(\sqrt{-23})$.  The connected components of $\Gamma$ display a wide variety of shapes, their structure being determined by the connections between solutions of the above Thue equation.  The connected component of the vertex $1$, displayed in Figure \ref{fig:2}, is the most complex of the connected components we have found.  \medskip

Finally, in Section 5 we give a proof that the map $f_{-29/16}(x)$ has no rational $n$-cycles for $n \neq 3$, and show more generally that any map $f_c(x)$ with a rational $3$-cycle, for which the denominator of $c$ is relatively prime to $3, 5$ or $7 \cdot 29$, has no rational $n$-cycles for $n \neq 3$.  (In \cite{ms} this is expressed by saying that the map $f_c(x)$ has {\it good reduction} at these primes.)  \medskip

It is hoped that the arithmetic properties proved here may eventually contribute to a proof that the quadratic maps with rational $3$-cycles have no other rational periodic cycles.

\section{Quadratic maps with rational $3$-cycles.}
A point $\alpha$ of minimal period $3$ for the map $f(x) = x^2+c$ satisfies the equations $f^3(\alpha) = \alpha$ and $f(\alpha) \neq \alpha$, where $f^n(x)$ denotes the $n$-fold iteration of $f(x)$ with itself.  Hence, setting $f(x,y) = x^2+y$, the point $(x,y) = (\alpha, c)$ satisfies the equation of the curve $\Phi_3(x,y) = 0$, where
\begin{align}
\notag \Phi_3(x,y) = &\frac{f^3(x,y) - x}{f(x,y)-x} = x^6 +x^5 +(3y+1)x^4 +(2y+1)x^3\\
\label{eqn:1} &+(3y^2 +3y+1)x^2 +(y^2 +2y+1)x+y^3 +2y^2 +y+1.
\end{align}
See \cite{m}, \cite{mp}.  The curve $\Phi_3(x,y) = 0$ has genus $0$ and the rational point $(x,y) = (-\frac{7}{4},-\frac{29}{16})$ and therefore has a rational parametrization, which can be given by
\begin{align}
\notag x = x(t) &= \frac{t^3 +t^2 - t+7}{4(t^2 - 1)},\\
y = y(t) &= -\frac{t^6-2t^5+11t^4+20t^3+23t^2-18t+29}{16(t^2-1)^2}.
\label{eqn:2}
\end{align}
Here $t$ is the parameter given by
$$t = 1 + 2(x^2 + x + y).$$
See \cite[Thm. 4]{m}.  In particular, with $t = 0$ we have (see \cite[eq. (2)]{m})
$$\Phi_3\left(x,-\frac{29}{16}\right) = \left(x+\frac{7}{4}\right)\left(x+\frac{1}{4}\right)\left(x-\frac{5}{4}\right)\left(x^3+\frac{1}{4}x^2-\frac{41}{16}x+\frac{23}{64}\right).$$
This parametrization is therefore ``centered" at $c = -\frac{29}{16}$. \medskip

An equivalent parametrization, given in \cite{rw}, can be easily derived, as follows.  Putting $s = x^2+x+y$, write the polynomial $\Phi_3(x,y)$ in terms of $x$ and $s$ and set it equal to zero:
$$\Phi_3(x,y) = s^3 - 2s^2 x + 2s^2 - 2sx + s + 1 = 0.$$
Then solve for $x = x_1(s)$ and $y = y_1(s) = s-x_1^2(s)-x_1(s)$ in terms of $s$:
\begin{align}
\label{eqn:8} x_1(s) & = \frac{s^3 + 2s^2 + s + 1}{2s(s+ 1)},\\
\label{eqn:9} y_1(s) & =  -\frac{s^6 + 2s^5 + 4s^4 + 8s^3 + 9s^2 + 4s + 1}{4s^2(s + 1)^2}.
\end{align}
(This calculation is implicit in the proof of \cite[Lemma 1]{m}.  Note that the equation $2(b + 1)c = 2(b - 1)$ in the last paragraph of that proof should read $2(b + 1)c = -2(b + 1)$ or $2(b+1)(c+1) = 0$, so that the putative root $ax^2+bx+c = -x^2 -x +c$ in that proof equals $y = -x^2-x+s$ if $c$ is replaced by $s$.)  Note that $y_1(s)$ is invariant under the map $s \rightarrow \psi(s) = -\frac{s+1}{s}$ and its square $s \rightarrow \psi^2(s) = \frac{-1}{s+1}$.  Applying the map $\psi(s)$ to $x_1(s)$ gives the other elements in the orbit of $x_1(s)$:
\begin{align}
\label{eqn:11} x_2(s) & = x_1(\psi(s)) = \frac{s^3 - s - 1}{2s(s + 1)},\\
\label{eqn:12} x_3(s) & = x_1(\psi^2(s)) = -\frac{s^3 + 2s^2 + 3s + 1}{2s(s + 1)}.
\end{align}
Throughout the paper, we will use $\{x_1, x_2, x_3\}$ to denote the unique rational orbit of the map $f_c(x) = x^2+c$, where $c = y(t)$ or $y_1(s)$ is an element of $\mathbb{Q}$.  Note that the parameter $s = x^2 + x + y$ is rational whenever $(x,y)$ is a rational point on $\Phi_3(x,y) = 0$.  Furthermore, only one of the two orbits of period $3$ can be rational, by \cite[Thm. 3]{m}.\medskip

We begin by proving the following result.

\newtheorem{thm}{Theorem}

\begin{thm}
The value $c = -\frac{29}{16}$ is the unique rational number with smallest denominator, for which the quadratic map $f_c(x) = x^2 + c$ has a rational cycle with period $3$.
\label{thm:1}
\end{thm}

\begin{proof}
Assume $c = y_1(s)$, so that $f_c(x)$ has a rational $3$-cycle.  Let $s = \frac{m}{n}$, where $m, n \in \mathbb{Z}$ and $(m,n) = 1$.  Then
\begin{equation}
y_1\left(\frac{m}{n}\right) = -\frac{m^6 + 2m^5n + 4m^4n^2 + 8m^3n^3 + 9m^2n^4 + 4mn^5 + n^6}{4m^2n^2(m + n)^2}.
\label{eqn:3}
\end{equation}
The numerator in this expression is
\begin{align}
\label{eqn:4} A(m,n) & = m^6 + 2m^5n + 4m^4n^2 + 8m^3n^3 + 9m^2n^4 + 4mn^5 + n^6\\
\notag & \equiv (m^3 + mn^2 + n^3)^2 \ (\textrm{mod} \ 2).
\end{align}
It is clear that $A(m,n)$ is always an odd integer and that $16$ divides the denominator of (\ref{eqn:3}) when written in lowest terms.  Now set
\begin{equation}
B(m,n) = 4m^2n^2(m + n)^2 = 16 \left(\frac{mn(m+n)}{2}\right)^2.
\label{eqn:4b}
\end{equation}
We have that $(A(m,n), mn) = (A(m,n),m+n) = 1$, the latter since $A(m,-m) = m^6$. Thus, any common prime factor of $A(m,n)$ and $m+n$ divides $m$ and the expression
$$y_1\left(\frac{m}{n}\right) = -\frac{A(m,n)}{B(m,n)}$$
in (\ref{eqn:3}) must be in lowest terms.  Now the denominator $B(m,n)$ can only equal $16$ when $mn(m+n) = \pm 2$.  Hence, the only possibilities are
$$(m,n) = (1, 1), (-1, -1), (2, -1), (-2, 1), (1, -2), (-1, 2),$$
yielding that $s \in \{ 1, -2, -\frac{1}{2}\}$.  Hence $c = y_1(1) = y_1(-2) = y_1(-1/2) = -\frac{29}{16}$.
\end{proof}
\medskip

\noindent {\bf Remark.} In the parametrization of (\ref{eqn:1}) given by (\ref{eqn:8}) and (\ref{eqn:9}), $y_1(s)$ is invariant under the map $s \rightarrow \psi(s)$ and its square.  This translates to the following invariance property for $A(m,n)$:
\begin{equation}
A(m,n) = A(-m-n,m) = A(-n,m+n).
\label{eqn:10}
\end{equation}
Later we will have occasion to use the mapping
\begin{equation}
\beta(m,n) = (-n, m+n),
\label{eqn:10a}
\end{equation}
which satisfies $\beta^3(m,n) = (-m, -n)$ and therefore has order $6$ on pairs $(m,n) \in \mathbb{Z}^2$ with $(m,n)=1$ and $mn(m+n) \neq 0$.  Equation (\ref{eqn:10}) shows that $A(\beta^i(m,n)) = A(m,n)$ for $0 \le i \le 5$.
\medskip

The argument in the proof of Theorem \ref{thm:1} is the basis for proving the following theorem, and for proving Theorems \ref{thm:3} and \ref{thm:4} below.  In the rest of the paper we will make use of the formula $c = y_1(m/n) = -\frac{A(m,n)}{B(m,n)}$ for the rational values of $c$, for which $f_c$ has a rational $3$-cycle.  The polynomials $A(m,n), B(m,n)$ will always mean the expressions in (\ref{eqn:4}) and (\ref{eqn:4b}).

\begin{thm} Define the expressions $A(m,n)$ and $B(m,n)$ by (\ref{eqn:4}) and (\ref{eqn:4b}). \smallskip

(a) If $m, n \in \mathbb{Z}$ with $\gcd(m,n) = 1$, the exact denominator of $c = y_1(m/n) = -\frac{A(m,n)}{B(m,n)}$ is
$$B(m,n) = 16 C(m,n)^2, \ \ \textrm{where} \ \ C(m,n) = \frac{mn(m+n)}{2} \in \mathbb{Z}.$$

(b) The values of $c \in \mathbb{Q}$ which are given by (\ref{eqn:2}) (or (\ref{eqn:9})) have the form
$$c =  -\frac{\textsf{Norm}_{\mathbb{Q}(\zeta)/\mathbb{Q}}(m-(\zeta+\zeta^2)n)}{(4C(m,n))^2},$$
where $\zeta = e^{2\pi i/7}$ is a primitive $7$-th root of unity and $m,n \in \mathbb{Z}$ with $(m,n) = 1$. \smallskip

\noindent (c) $c=-\frac{29}{16}$ is the only rational number whose denominator is a power of $2$, for which $f_c(x) = x^2+c$ has a rational cycle of period $3$. \smallskip

\noindent (d) If $c$ is the only rational number with a given denominator, for which $f_c(x)$ has a rational cycle of period 3, then $c = -\frac{29}{16}$. \smallskip

\noindent (e) If the numerator $A(m,n) = q^e$ of $c = -\frac{A(m,n)}{B(m,n)}$ is a prime power, where $q \neq 7$, then this is the only rational $c$ with numerator $q^e$, for which $f_c(x)$ has a rational $3$-cycle.  The same holds if $A(m,n) = 7q^e$, for some prime power $q^e$.
\label{thm:2}
\end{thm}

\begin{proof}
Part (a) follows from the proof of Theorem \ref{thm:1}.  Part (b) follows from (\ref{eqn:4}) and the fact that
$$f_\theta(x) = x^6 + 2x^5 + 4x^4 + 8x^3 + 9x^2 + 4x + 1 = \textsf{Norm}_{\mathbb{Q}(\zeta)/\mathbb{Q}}(x-(\zeta+\zeta^2))$$
is the minimal polynomial of $\theta = \zeta + \zeta^2$ over $\mathbb{Q}$, so that
\begin{equation}
A(m,n) = \textsf{Norm}_{\mathbb{Q}(\zeta)/\mathbb{Q}}(m-(\zeta+\zeta^2)n).
\label{eqn:4c}
\end{equation}
For (c), $B(m,n)= 4m^2n^2(m + n)^2$ can only be a power of $2$ if $m$ and $n$ are powers of $2$, which implies, without loss of generality, that $m = 2^k$ and $n = \pm 1$.  But $m+n = 2^k \pm 1$ must be a power of $2$, as well, which gives that $k=0$ and $n = 1$ or $k = 1$ and $n = -1$.  Thus, $B(m,n) = 16$, giving $c = -29/16$ by Theorem \ref{thm:1}.  For (d), note that $c = -A(m,n)/B(m,n)$ has the same denominator as $c' = -A(n,m)/B(n,m)$, and $c \neq c'$ unless $A(m,n) = A(n,m)$.  This condition holds if and only if
$$A(m,n)-A(n,m) = -mn(m - n)(2m + n)(m + 2n)(m + n) = 0.$$
Hence, $c \neq c'$ unless $m/n = 1, -1/2, -2$, i.e., $c = -29/16$.  Thus, if $c \neq -29/16$, there is another rational number $c'$ having the same denominator as $c$, for which $f_{c'}(x)$ has a rational cycle of period $3$. \smallskip

To prove (e), note that the field $\mathbb{Q}(\zeta)$ contains $\mathbb{Q}(\sqrt{-7})$, so by (b) the numerator $A(m,n)$ can be written as
\begin{align}
\notag A(m,n) &= \textsf{N}_{\mathbb{Q}(\sqrt{-7})/\mathbb{Q}}\left((m-n(\zeta+\zeta^2))(m-n(\zeta^2+\zeta^4))(m-n(\zeta+\zeta^4))\right)\\
\label{eqn:5} & = (m^3+m^2n-2m n^2-n^3)^2+7m^2n^2(m+n)^2.
\end{align}
If $A(m,n) = q^e$, $q \neq 7$, is a prime power, then $q^e = x^2+7y^2$ for some $x,y \in \mathbb{Z}$ with $(x,y) = (m^3+m^2n-2m n^2-n^3, mn(m+n)) = 1$.  Furthermore, $x = m^3+m^2n-2m n^2-n^3 \neq 0$ since this cubic form is irreducible and $y = mn(m+n) \neq 0$ by assumption.  This implies that the prime $q = \pi \bar \pi$ splits in the field $\mathcal{Q}=\mathbb{Q}(\sqrt{-7})$, where the norms of the conjugate primes $\pi, \bar \pi$ are both equal to $q$.  Then $q^e = x^2+7y^2$ implies that $x+y \sqrt{-7} = \pm \pi^i \bar \pi^j$, for some $i, j \ge 0$, $i+j = e$.  If $i$ and $j$ were both positive, with $i \le j$, this would imply that $\pi^i \bar \pi^i = q^i$ would divide $x+y \sqrt{-7}$.  But $q$ is odd and $\{1, \frac{1+\sqrt{-7}}{2}\}$ is a basis for the ring of integers in this field; hence $q^i \mid (x,y)$.  Thus, one of $i$ or $j$ must be $0$.  In that case, $x+y \sqrt{-7} = \pm \pi^e$, say.  Then the only other possible solutions of $A(m',n') = q^e = x'^2+7y'^2$ would be with $x'+y'\sqrt{-7} = \pm \bar \pi^e = \pm (x-y\sqrt{-7})$.  This gives that $m'n'(m'+n') = y' = \pm y = \pm mn(m+n)$, showing that $c = -\frac{A(m,n)}{B(m,n)}$ is unique. \medskip

If $A(m,n) = 7q^e = x^2 + 7y^2$, then a similar argument shows that the only possible solutions of  $A(m,n) = 7q^e$ satisfy $x+y\sqrt{-7} = \pm \sqrt{-7} \pi^e$ or $\pm \sqrt{-7} \bar \pi^e$ and that $y$ is unique, up to sign.
\end{proof}

\noindent {\bf Remarks.}
1. The proof of Theorem \ref{thm:2}(e) makes it clear that $y(t) < 0$ for all $t \in \mathbb{Q}, \ t \neq \pm 1$, since $A(m,n) = x^2+7y^2 > 0$.  Hence, $f_c(x) = x^2 + c$ has no rational periodic points of period $3$ if $c \ge 0$, by the results of \cite{m}. \smallskip

\noindent 2. The polynomial in (\ref{eqn:5}),
$$s(m,n) = m^3+m^2n-2m n^2-n^3 = \textsf{N}_{\mathbb{Q}(\theta)/\mathbb{Q}}(m-n(\zeta+\zeta^6)),$$
represents a norm to $\mathbb{Q}$ from the cyclic cubic extension $\mathbb{Q}(\theta)$, where $\theta = \zeta+\zeta^6 =2 \cos(2\pi/7)$. Also, $s(-n,m+n) = -s(m,n)$.  \smallskip

\noindent 3. The argument in part (e) of the above proof also shows that $A(m,n) =x^2+7y^2$ cannot be divisible by $7^e$, for $e \ge 2$, since this would imply that $7 \mid (x,y)$. \smallskip

\noindent 4. The numbers $C(m,n) = \frac{mn(m+n)}{2}$ are a generalization of triangular numbers:
$$C(m,n)=\sum_{i=1}^m\sum_{j=1}^n(i+j-1).$$

\newtheorem{conj}{Conjecture}

We suspect that the statement in Theorem \ref{thm:2}(e) is true more generally.

\begin{conj} If $A(m,n) = k$, for some pair $(m,n) \in \mathbb{Z}^2$ with $mn(m+n) \neq 0$ and $\gcd(m,n) = 1$, then $c = -\frac{A(m,n)}{B(m,n)}$ is the only rational number having the numerator $k$, for which $f_c(x)$ has a rational $3$-cycle.  Equivalently, if $k \in \mathbb{N}$, the equation
$$A(m,n) = m^6 + 2m^5n + 4m^4n^2 + 8m^3n^3 + 9m^2n^4 + 4mn^5 + n^6 = k$$
has either $0$ or $6$ solutions $(m,n)$, for which $mn(m+n) \neq 0, \gcd(m,n) = 1$.
\label{conj:1}
\end{conj}

We have checked this on Pari for $k \le 10^{15}$, by verifying that the equation $A(m,n) = k$ has exactly six solutions for values of $k$ in this range.  These solutions are obtained from one solution $(m,n)$ using powers of the transformation $\beta(m,n) = (-n, m+n)$.  Cf. (\ref{eqn:10}).  It is of course clear that $A(m,n) = k$ has at most a finite number of solutions, by the Thue-Siegel-Roth theorem \cite{la1}, \cite{la2}.  In our case this is easy to see directly from equation (\ref{eqn:5}).  For a given positive integer $k = x^2+7y^2$, there are only finitely many possible values of $y = \pm mn(m+n)$, and therefore only finitely many possibilities for $m$ and $n$. \medskip

To prove this conjecture, one would need to show that two distinct elements $m_1+n_1(\zeta+\zeta^2)$ and $m_2+n_2(\zeta+\zeta^2)$ in the $\mathbb{Z}$-module $\mathbb{Z}[1,\zeta+\zeta^2]$ do not have the same norm from the field $\mathbb{Q}(\zeta)$, when $(m_2,n_2)$ is different from any of the transforms $\beta^i(m_1,n_1)$, for $0 \le i \le 5$. \medskip

This conjecture can also be reformulated using (\ref{eqn:5}) as follows. \medskip

\noindent {\bf Conjecture 1$'$.} {\it If the integer $k$ has two representations $k = X_i^2+7Y_i^2$, $i = 1,2$, for which
$$X_i =  m_i^3+m_i^2n_i-2m_i n_i^2-n_i^3, \ \ Y_i = m_in_i(m_i+n_i),$$
for pairs $(m_i, n_i)$ satisfying $m_in_i(m_i+n_i) \neq 0$ and $\gcd(m_i, n_i) = 1$, then necessarily $Y_1 = \pm Y_2$.}

\section{Arithmetic properties of the maps $f_c(x)$.}

\begin{thm} If $m,n \in \mathbb{Z}$ with $(m,n) = 1$, the numerator of $c = -\frac{A(m,n)}{B(m,n)} \in \mathbb{Q}$ has the form $A(m,n) = 7^a (14b+1)$, for $a, b \in \mathbb{N}$.  Moreover, all the prime factors of the integer $14b+1$ are congruent to $1$ (mod $7$).  If $p$ is any prime satisfying $p \equiv 1$ (mod 7), then there are infinitely many values of $c = -\frac{A(m,n)}{B(m,n)}$ for which $p \mid A(m,n)$ and $f_c(x)$ has a rational $3$-cycle.
\label{thm:3}
\end{thm}

\begin{proof} It is well-known that the ring $\textsf{R}$ of integers in the field $F = \mathbb{Q}(\zeta)$ has unique factorization and that $7 = \varepsilon (1-\zeta)^6$ is a power of the prime $\pi = 1-\zeta$ in $\textsf{R}$, where $\varepsilon \in \textsf{R}^\times$ is a unit.  See \cite[p. 2]{wa} for the latter fact and \cite[pp. 570, 590]{h} for unique factorization.  \medskip

Noting (\ref{eqn:4c}) and factoring the integer $m-(\zeta+\zeta^2)n$ in $\textsf{R}$ gives
$$m-(\zeta+\zeta^2)n = (1-\zeta)^a \gamma, \ \ \gamma \in \textsf{R}, \ a \ge 0,$$
where $\gamma$ is relatively prime to $1-\zeta$.  Writing $\textsf{N}$ for $\textsf{Norm}_{\mathbb{Q}(\zeta)/\mathbb{Q}}$ gives
$$\textsf{N}(m-(\zeta+\zeta^2)n) = \textsf{N}((1-\zeta)^a) \textsf{N}(\gamma) = 7^a \textsf{N}(\gamma).$$
It is well-known that norms to $\mathbb{Q}$ of elements of $\textsf{R}$ prime to $7$ are congruent to $1$ (mod $7$).  (See \cite[Ch. 2]{wa}.)  This can be seen easily as follows, using the fact that $\zeta \equiv 1$ (mod $\pi$) and that powers of $\zeta$ form an integral basis of $\textsf{R}$.  If $\gamma = \sum_{i=1}^6{a_i \zeta^i}$, with $a_i \in \mathbb{Z}$, then letting $A = \sum_{i=1}^6{a_i}$ gives that
\begin{equation*}
\textsf{N}(\gamma) = \prod_{j=1}^6{(\sum_{i=1}^6{a_i \zeta^{ij}})} \equiv \prod_{j=1}^6{(\sum_{i=1}^6{a_i})} = A^6  \equiv 1 \ (\textrm{mod} \ \pi)
\end{equation*}
in $\textsf{R}$, since $A^6 \equiv 1$ (mod $7$).  Now, $\textsf{N}(\gamma) \equiv 1$ (mod $\pi$) implies that $\textsf{N}(\gamma) \equiv 1$ (mod $7$).  Since $\textsf{N}(\gamma)$ is positive and odd, this proves the first assertion.  \medskip

To prove the second assertion, let $p \mid A(m,n)$, where $(m,n) = 1$.  Then $p \nmid n$ and $m/n \equiv k$ (mod $p$), for some $k \in \mathbb{Z}$.  Recalling the polynomial $f_\theta(x)$ from the proof of Theorem \ref{thm:2}, it follows that
$$A(m,n) = n^6 f_\theta(m/n) \equiv n^6 f_\theta(k) \ (\textrm{mod} \ p),$$
so that $p \mid f_\theta(k)$.  This shows that $p$ is a so-called {\it prime divisor} of the polynomial $f_\theta(x)$, since $p$ divides a value of $f_\theta(x)$, for some $x \in \mathbb{Z}$.  See \cite{gb}.  Now note that $\mathbb{Q}(\theta) = \mathbb{Q}(\zeta)$ and $\textrm{disc}(f_\theta(x)) = -2^6 7^5$.  It follows from \cite[Theorem 2, Lemma 1]{gb} that the prime divisors of $f_\theta(x)$ are the same as the prime divisors of the $7$-th cyclotomic polynomial
$$\Phi_7(x) = x^6+x^5+x^4+x^3+x^2+x+1,$$
with the possible exception of the prime $p = 2$.  However, it is clear that $2$ is not a prime divisor of either polynomial, so these two polynomials have exactly the same set of prime divisors.  Now Prop. 2.10 of \cite{wa} shows that $p \neq 7$ is a prime divisor of $f_\theta(x)$ if and only if $p \equiv 1$ (mod $7$).  If $k \in \mathbb{Z}^+$ is a root of $f_\theta(x) \equiv 0$ (mod $p$), then there are infinitely many pairs $(m,n)$ satisfying $m \equiv kn$ (mod $p$) and $\gcd(m,p) = \gcd(n,p) =\gcd(m,n) = 1$, with $mn(m+n) \neq 0$.  For example, choose an integer $n > 0$ for which $p \nmid n$ and let $m = kn + lp$ for any integer $l > 0$ with $(l,n) = 1$.  This gives that $p \mid A(m,n)$ by the above congruence.
\end{proof}

\begin{thm} The smallest possible value of the numerator $A(m,n)$ in the representation
$$c = -\frac{A(m,n)}{B(m,n)}$$
is $A(m,n) = 29$, and this numerator occurs only for $c = -\frac{29}{16}$.
\label{thm:4}
\end{thm}

\begin{proof} First note that $29$ is the smallest prime number which is congruent to 1 (mod $7$).  Theorem \ref{thm:3} implies that the smallest possible numerator $A(m,n)$ is either $1, 7$ or $29$, so we just have to exclude $1$ and $7$ as possibilities.  By Theorem \ref{thm:2}(b) and the first part of the proof of Theorem \ref{thm:3}, these are only possible if
$$\textsf{N}(m-n(\zeta+\zeta^2)) = 1 \ \textrm{or} \ 7.$$
This happens if and only if
\begin{equation}
m-n(\zeta+\zeta^2) = \varepsilon \ \ \textrm{or} \ \ m-n(\zeta+\zeta^2) = \varepsilon (1-\zeta), \ \ \varepsilon \in \textsf{R}^\times,
\label{eqn:6}
\end{equation}
for some unit $\varepsilon$.  We need the fact that the units in $F = \mathbb{Q}(\zeta)$ are generated by $\zeta$ and the units in the real cubic subfield $F^+ = \mathbb{Q}(\zeta+\zeta^{-1})$.  (See \cite[Prop. 1.5]{wa}.)  Thus, we can write $\varepsilon = \zeta^a \varepsilon_0$, with $\varepsilon_0 \in F^+$ and $a \in \mathbb{Z}$.  To show that (\ref{eqn:6}) is impossible, it suffices to show that for any integer $a$, with $0 \le a \le 6$, we have
\begin{equation}
\zeta^{-a} (m-n(\zeta+\zeta^2)) \notin F^+ \ \ \textrm{or} \ \ \frac{m-n(\zeta+\zeta^2)}{\zeta^a(1-\zeta)} \notin F^+,
\label{eqn:7}
\end{equation}
respectively.  First, assume the negation of the first statement in (\ref{eqn:7}).  Then taking the complex conjugate gives
$$\zeta^{-a} (m-n(\zeta+\zeta^2)) = \zeta^a(m-n(\zeta^5+\zeta^6)),$$
or
$$m-n \zeta -n\zeta^2 = \zeta^{2a}(m-n\zeta^5 -n\zeta^6).$$
Now we use the fact that a basis for $F/\mathbb{Q}$ is $\{1, \zeta, \zeta^2, \zeta^3, \zeta^4, \zeta^5\}$.  We eliminate the possible values of 
$a$ by equating coefficients of a suitable power of $\zeta$.  For example, if $a=1$, we obtain
$$m-n \zeta -n\zeta^2 = -n -n\zeta + m\zeta^2,$$
giving that $m=-n$, which is excluded.  Similarly, for each of the other values of $a$, we conclude that $n=0, m = -n$ or $m=0$, all of which are excluded by Theorem \ref{eqn:2}(a).  This proves the first statement in (\ref{eqn:7}). \medskip

Now assume the negation of the second assertion in (\ref{eqn:7}).  Then
$$\frac{m-n(\zeta+\zeta^2)}{\zeta^a(1-\zeta)} = \frac{m-n(\zeta^5+\zeta^6)}{\zeta^{-a}(1-\zeta^6)}$$
or
$$(m-n(\zeta+\zeta^2))(1-\zeta^6) = \zeta^{2a}(m-n(\zeta^5+\zeta^6))(1-\zeta).$$
Multiplying out and using $\Phi_7(\zeta) = 0$ gives
$$2m+n +m\zeta +(m-n)\zeta^2 +m\zeta^3+m\zeta^4+m\zeta^5 = \zeta^{2a}(m+n -m\zeta-n\zeta^5).$$
Substituting for $a \in \{0, 1, ..., 6\}$ and writing the right side in terms of the basis $\{\zeta^i\}$ implies in each case that $m = 0, n= 0$ or $m=-n$.  For example, if $a=2$, the last equation becomes
$$2m+n +m\zeta +(m-n)\zeta^2 +m\zeta^3+m\zeta^4+m\zeta^5 = -n\zeta^2+(m+n)\zeta^4 -m\zeta^5,$$
giving that $m = 0$.  The other cases are similar.  This proves the second assertion in (\ref{eqn:12}).  Therefore, both cases in (\ref{eqn:6}) are impossible.  \medskip

The assertion about $A(m,n) = 29$ occurring only for $c=-29/16$ is immediate from Theorem \ref{thm:2}(e).
\end{proof}

\noindent {\bf Remark.}
Note that $\zeta+\zeta^2$ is a unit in $\mathbb{Q}(\zeta)$.  By Remark 3 after Theorem \ref{thm:2} and the above proof, the numerator $A(m,n)$ is never a power of $7$.  Thus, $A(m,n)$ is always divisible by at least one prime $p \equiv 1$ (mod $7$).
\medskip

The above results show that the value $c = -29/16$ is minimal, with respect to both the denominator and the numerator, for maps $f_c(x)$ having a rational $3$-cycle.  This shows that $c$ is the rational number of smallest height, for which $f_c(x)$ has a rational $3$-cycle.  We emphasize the important role that the prime $7$ plays in these results, which is also a feature of the discussion in \cite{m}.  The value $c = -29/16$ also plays a prominent role in \cite{p}.  See \cite[Thm. 3]{p}. \medskip

Taking a hint from the numerator of (\ref{eqn:11}), let $\gamma$ be the real root of $x^3-x-1$.  Then $\textrm{disc}(x^3-x-1) = -23$ and
$$\gamma = \frac{1}{6}\left(\sqrt[3]{108+12\sqrt{69}}+\sqrt[3]{108-12\sqrt{69}}\right),$$
and the quantities $-\gamma^2$ and $\gamma - \gamma^2$ have the respective minimal polynomials
$$f_{-\gamma^2}(x) = x^3+2x^2+x+1 \ \ \textrm{and} \ \ f_{\gamma-\gamma^2}(x) = x^3+2x^2+3x+1.$$
Now let $K = \mathbb{Q}(\gamma)$ be the real cubic field of discriminant $d=-23$.  The following theorem shows that the numerators of the rational numbers $x_i(s)$, for $s \in \mathbb{Q}$, satisfy a similar arithmetic condition to the condition on $c$ given by Theorem \ref{thm:2}(b), but with an interesting twist.

\begin{thm}
(a) If $s = \frac{m}{n}$ with $(m,n) = 1$ and $c = -\frac{A(m,n)}{B(m,n)}$, then the $3$-cycle of the map $f_c(x)$ is given by
\begin{align*}
x_1\left(\frac{m}{n}\right) &= \frac{\textsf{N}_{K/\mathbb{Q}}(m+n\gamma^2)}{2mn(m+n)},\\
x_2\left(\frac{m}{n}\right) &= \frac{\textsf{N}_{K/\mathbb{Q}}(m-n\gamma)}{2mn(m+n)},\\
x_3\left(\frac{m}{n}\right) &= -\frac{\textsf{N}_{K/\mathbb{Q}}(m-n(\gamma-\gamma^2))}{2mn(m+n)},
\end{align*}
where $\textsf{N}_{K/\mathbb{Q}}$ denotes the norm to $\mathbb{Q}$ from the real cubic field $K = \mathbb{Q}(\gamma)$ and these rational numbers are in lowest terms. \smallskip

(b) If $s \in \mathbb{Q}$ and $q$ is a prime factor of the numerator of $x_i(s)$, then $q$ is an odd prime satisfying either 
$$(i) \ q =23,$$
$$(ii) \ \left(\frac{-23}{q}\right) = -1,$$
or
$$(iii) \ \left(\frac{-23}{q}\right) = +1 \ \textrm{and} \ q = x^2+23y^2, \ \ x,y \in \mathbb{Z}.$$
If $q$ is any prime satisfying one of these conditions, then there are infinitely many values of $c = -\frac{A(m,n)}{B(m,n)}$ for which $q$ divides the numerator of one of the rational numbers $x_i$ in the rational $3$-cycle of $f_c$.
\label{thm:5}
\end{thm}

\begin{proof}
The formulas in (a) follow directly from (\ref{eqn:8}), (\ref{eqn:11}) and (\ref{eqn:12}) and the minimal polynomials given above.  The fact that these formulas give the values of the $x_i$ in lowest terms follows exactly as in the proof of Theorem \ref{thm:1}.  To prove (b), we use the fact that the field $K$ has discriminant $-23$ and class number $1$, so its ring of integers $\textsf{R}_K$ has unique factorization.  Moreover, $\Sigma = K(\sqrt{-23})$ is the Hilbert class field (see \cite[pp. 94-95]{c}) of the quadratic field $\mathbb{Q}(\sqrt{-23})$.  Note that $\{1, \gamma, \gamma^2\}$ is an integral basis for $\textsf{R}_K/\mathbb{Z}$.  Hence the numbers $m+n\gamma^2, m-n(\gamma-\gamma^2), m-n\gamma$ are all primitive, meaning that they are not divisible by any rational prime $q$.  Otherwise $q$ would divide $\gcd(m,n)$.  \medskip

Now the primes $\pi$ in $\textsf{R}_K$ come in several varieties. \smallskip

\noindent 1. The prime $23$ factors in $\textsf{R}_K$ as
$$23 = \varepsilon (\gamma-3)(2\gamma+3)^2, \ \ \varepsilon = \gamma-\gamma^2 \in \textsf{R}_K^\times,$$
and is divisible by two primes, both of whose norms are $\pm 23$. \smallskip

\noindent 2.  If $\left(\frac{-23}{q}\right) = -1$, then the rational prime $q$ factors into two primes $q = \pi_1 \pi_2$, where $\textsf{N}(\pi_1) = \pm q$ and $\textsf{N}(\pi_2) = \pm q^2$. \smallskip

\noindent 3. If $\left(\frac{-23}{q}\right) = +1$ and $q = x^2+23y^2$ for some $x, y \in \mathbb{Z}$, then $q = \pi_1 \pi_2 \pi_3$ factors into three primes, each with norm $\textsf{N}(\pi_i) = \pm q$.  \smallskip

\noindent 4. Finally, if $\left(\frac{-23}{q}\right) = +1$ and $q \neq x^2+23y^2$, for any $x, y \in \mathbb{Z}$, then $q$ remains prime in $\textsf{R}_K$. \medskip

The connection with the Legendre symbol can be seen as follows.  A theorem of Pellet-Stickelberger-Voronoi (originally proved by Pellet but usually attributed just to Stickelberger \cite[p. 485]{h}) says that if $q \neq 23$ and the polynomial $f_\gamma(x) = x^3 -x-1$ has $r_q$ irreducible factors modulo $q$, then
$$\left(\frac{d}{q}\right) = \left(\frac{-23}{q}\right) = (-1)^{3+r_q}.$$
Thus, $f_\gamma(x)$ has two irreducible factors mod $q$ in Case 2 above and either $1$ or $3$ irreducible factors in Cases 3 and 4.  Then a theorem of Dedekind \cite[Prop. 2.14]{wa}, together with the fact that all ideals in $\textsf{R}_K$ are principal, shows that $q$ factors into primes in the same way that $f_\gamma(x)$ factors into irreducibles mod $q$.  If $f_\gamma(x)$ is a product of three linear factors (mod $q$), then $q$ is a product of three primes and Case 3 holds; while if $f_\gamma(x)$ is irreducible (mod $q$), then $q$ is a prime in $\textsf{R}_K$ and Case 4 holds. \medskip

The connection of the above facts with the quadratic form $x^2+23y^2$ follows from a theorem of complex multiplication.  In Case 3 above, the prime ideal $(q)$ of $\mathbb{Z}$ splits completely in $K$ and in $\mathbb{Q}(\sqrt{-23})$, and therefore splits completely in the composite field $\Sigma = K \cdot \mathbb{Q}(\sqrt{-23})$.  Hence, by the defining property of the Hilbert class field \cite[p. 98]{c}, the prime ideal factors of $(q)$ in $\mathbb{Q}(\sqrt{-23})$ must be principal, which implies that the norm of some integer in the field equals $q$, i.e. $x^2+23y^2 = q$.  The converse also holds, and this verifies the distinction between Cases 3 and 4.  See \cite[pp. 88, 98]{c}. \medskip

Now the primitivity of the algebraic integers $m+n\gamma^2, m-n(\gamma-\gamma^2), m-n\gamma$ implies that none can be divisible by a prime $q$ satisfying point 4.  Hence the prime divisors $q$ of the numerators of the $x_i(m/n)$ satisfy (i), (ii), or (iii).  Conversely, if $q$ is a prime satisfying (i), (ii) or (iii), then the polynomial $f_\gamma(x)$ has a linear factor $x-k$ (mod $q$). If $m \equiv kn$ (mod $q$) with $(m,q) = (n,q) = (m,n) = 1$, as in the last paragraph of the proof of Theorem \ref{thm:3}, it follows that $f_\gamma(m/n) \equiv 0$ (mod $q$).  Hence, $q$ divides the numerator of $x_2(m/n)$ for infinitely many pairs $(m,n)$.
\end{proof}

Note that the denominators of the rational numbers $x_i(m/n)$ in Theorem \ref{thm:5} are all $4C(m,n)$, where $C(m,n)$ is defined in Theorem \ref{thm:2}(a). \medskip

\noindent {\bf Examples.} For $c=-\frac{29}{16}$, the rational $3$-cycle is $\{-\frac{7}{4}, -\frac{1}{4}, \frac{5}{4}\}$, where
$$\left(\frac{-23}{5}\right) = \left(\frac{-23}{7}\right) = -1.$$
For $c = -\frac{43^2}{24^2}$, the rational $3$-cycle is $\{\frac{23}{24}, -\frac{55}{24}, \frac{49}{24}\}$, where
$$\left(\frac{-23}{11}\right) = -1.$$
For $c = -\frac{71 \cdot 2311}{2^8 \cdot 7^2}$, the rational $3$-cycle is
$$\left\{\frac{5 \cdot 67}{112}, -\frac{463}{112}, \frac{449}{112}\right\},$$
where
$$\left(\frac{-23}{67}\right) = -1, \ \left(\frac{-23}{463}\right) = \left(\frac{-23}{449}\right) = +1;$$
and
$$463 = 2^8 +23 \cdot 3^2, \ \ 449 = 3^4 + 23 \cdot 4^2.$$
We will show below in Theorems \ref{thm:13}, \ref{thm:14} and Corollary \ref{cor:4} that these $3$-cycles are the only rational cycles for their respective maps $f_c(x)$.  The proof of Theorem \ref{thm:14} shows that if $5 \nmid B(m,n)$, then $c = -\frac{A(m,n)}{B(m,n)} \equiv 1$ (mod $5$).  In this case $5$ always divides the numerator of one of the elements $x_i(s)$ of the $3$-cycle, as is illustrated by the above examples.  Note that the primes $3, 13, 29$ never occur as factors of the numerators of the $3$-periodic points $x_i$, by Theorem \ref{thm:5}(b), since $\left(\frac{-23}{3}\right) = \left(\frac{-23}{13}\right) = \left(\frac{-23}{29}\right) = +1$ and $3, 13, 29$ are not expressible in the form $x^2+23y^2$.  See Figure \ref{fig:5} below for a list of examples for small values of $m$ and $n$.

\begin{figure}[!htb]
\begin{center}
\includegraphics[width=150mm,scale=2.50]{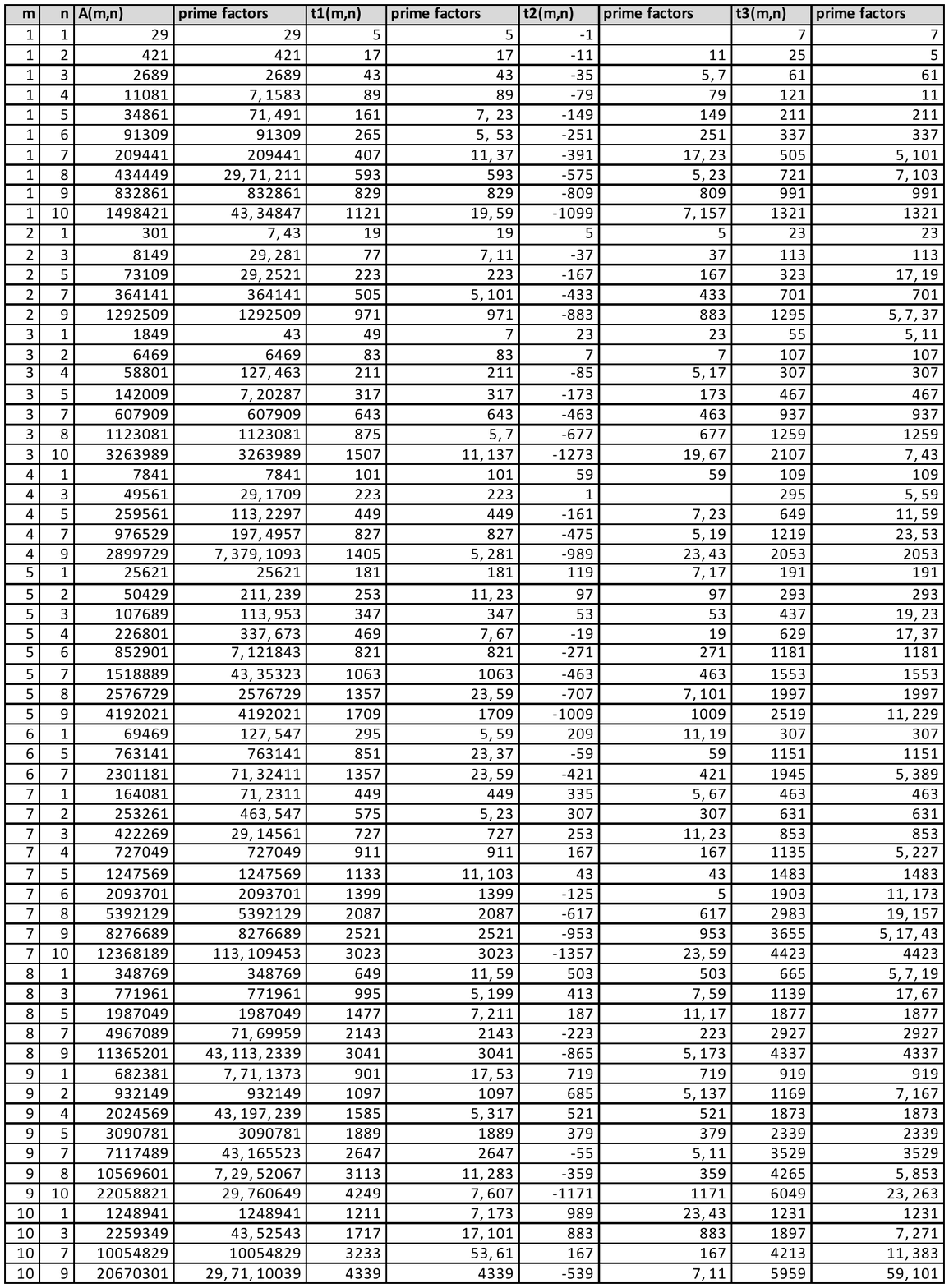}
\end{center}
\caption{\label{fig:5} Prime factors of numerators of $c=-\frac{A(m,n)}{B(m,n)}$ and $x_i(m,n) = \pm \frac{t_i(m,n)}{4C(m,n)}$ for small $m$ and $n$.  (See eqs. (\ref{eqn:20})-(\ref{eqn:22}).)}
\end{figure}

\newtheorem{cor}{Corollary}

\begin{cor} If $m,n \in \mathbb{Z}$ with $(m,n) = 1$, then we have the following norm equation:
\begin{equation}
\textsf{N}_{K/\mathbb{Q}}(m+n\gamma^2)^2-\textsf{N}_{F/\mathbb{Q}}(m-n(\zeta+\zeta^2)) = 4C(m,n) \textsf{N}_{K/\mathbb{Q}}(m-n\gamma).
\label{eqn:13}
\end{equation}
Here $F = \mathbb{Q}(\zeta)$ is the field of $7$-th roots of unity and $K = \mathbb{Q}(\gamma)$ is the real field generated by a root of $x^3-x-1=0$.
\label{cor:1}
\end{cor}

\begin{proof} This follows from $x_1^2(s)+c = x_2(s)$, Theorem \ref{thm:2}(b) and Theorem \ref{thm:5}(a), on multiplying through by $(4C(m,n))^2$.
\end{proof}

\noindent {\bf Remark.} It is easy to see, using (\ref{eqn:10}) and the relations
\begin{align}
\label{eqn:14} -m-n+m\gamma^2 &= (\gamma^2-1)(m-n\gamma), \ \ \gamma^2-1 = \gamma^{-1},\\
\notag -m-n-m\gamma &= -(\gamma+1)(m-n(\gamma-\gamma^2)),\\
\label{eqn:15} -n + (m+n)\gamma^2 & = \gamma^2(m-n(\gamma-\gamma^2)),
\end{align}
that the substitution $\beta^2(m,n) = (-m-n,m)$ transforms (\ref{eqn:13}) into
$$\textsf{N}_{K/\mathbb{Q}}(m-n\gamma)^2-\textsf{N}_{F/\mathbb{Q}}(m-n(\zeta+\zeta^2)) = -4C(m,n) \textsf{N}_{K/\mathbb{Q}}(m-n(\gamma-\gamma^2));$$
and $\beta(m,n) = (-n,m+n)$ transforms (\ref{eqn:13}) into
$$\textsf{N}_{K/\mathbb{Q}}(m-n(\gamma-\gamma^2))^2-\textsf{N}_{F/\mathbb{Q}}(m-n(\zeta+\zeta^2)) = 4C(m,n) \textsf{N}_{K/\mathbb{Q}}(m+n\gamma^2).$$
These last two equations are equivalent to $x_i^2(s)+c = x_{i+1}(s)$, for $i = 2, 3$, respectively, where $x_4 = x_1$.  All three of these norm equations are equivalent to polynomial identities, and also hold for $m,n \in \mathbb{Q}$. \medskip

It is curious that the numerators of the values $c = y(t) = y_1(s)$ are norms from the abelian field $\mathbb{Q}(\zeta)$, while the numerators of the elements $x_i = x_i(s)$ of the rational $3$-cycle of $f_c(x)$ are norms from the {\it non-abelian} extension $K$.  Note that the normal closure of $K/\mathbb{Q}$ is the field $\Sigma = K(\sqrt{-23})$, whose Galois group is $\textrm{Gal}(\Sigma/\mathbb{Q}) \cong D_3$. \medskip

\begin{thm}
(a) If $c = -\frac{A(m,n)}{B(m,n)}$, no two of the elements $x_i(s)$ in the rational $3$-cycle of $f_c(x)$ are divisible by the same prime.  In other words, their numerators are relatively prime in pairs.  \smallskip

(b) If $p$ is any prime dividing the numerator of one of the $x_i(s)$, then $c^3+2c^2+c+1 \equiv 0$ (mod $p$). \smallskip

(c) No more than one numerator of an $x_i(s)$ can be $\pm 1$.
\label{thm:6}
\end{thm}

\begin{proof}
The rational numbers in the orbit $\{x_1(s), x_2(s), x_3(s)\}$ of $f_c(x)$ are distinct and have the same denominator, so at most two of them can have a numerator equal to $\pm 1$.  It follows that at least one of these numerators is divisible by an odd prime $p$.  A prime $p$ cannot divide the numerators of two of the $x_i(s)$, because $p$ does not divide the resultant
\begin{equation*}
\textrm{Res}(s^3+2s^2+s+1,s^3-s-1) = -8.\\
\end{equation*}
See equations (\ref{eqn:8}) and (\ref{eqn:11}).  This implies, for example that
\begin{align*}
\textrm{Res}_m&(m^3+2m^2n+mn^2+n^3,m^3-mn^2-n^3) = -8n^9,\\
\textrm{Res}_n&(m^3+2m^2n+mn^2+n^3,m^3-mn^2-n^3) = 8m^9.
\end{align*}
The same calculation applies to the resultants of the other combinations of numerators, or can be deduced from this using the substitutions $\beta^2(m,n) = (-m-n,m)$ and $\beta(m,n) = (-n,m+n)$.  Furthermore, if some $x_i(s) \equiv 0$ (mod $p$), then the orbit is $\{0, c, c^2+c\}$ (mod $p$), which implies that $c$ satisfies
$$(c^2+c)^2+c \equiv c(c^3+2c^2+c+1) \equiv 0 \ (\textrm{mod} \ p).$$
If $c \equiv 0$ (mod $p$), then $0$ would be a fixed point (mod $p$), contradicting what we have just shown.  Therefore, $c \not \equiv 0$ (mod $p$) and $p \mid c^3+2c^2+c+1$.  This proves (a) and (b).  Note that (b) also follows immediately from (\ref{eqn:1}). \medskip

To prove (c), assume $x_i(s)$ has numerator equal to $\pm 1$ for two consecutive values of $i$ (mod $3$).  Then $x_i^2(s)+c = x_{i+1}(s)$ implies, on multiplying through by $4^2C(m,n)^2$, that
$$1-A(m,n) = \pm 4C(m,n).$$
But $x_{i+1}^2(s)+c = x_{i-1}(s)$ cannot also have numerator $\pm 1$, which shows that $1-A(m,n)$ must have a prime factor which does not divide $4C(m,n)$.  Thus, at most one $x_i(s)$ can have numerator $\pm 1$.
\end{proof}

\noindent {\bf Example.} The map $f(x) = x^2-\frac{29 \cdot 1709}{2^6 \cdot 3^2 \cdot 7^2}$, corresponding to $(m,n) = (4,3)$ in the table in Figure \ref{fig:5}, has the $3$-cycle
$$\left\{\frac{1}{168}, \frac{-5 \cdot 59}{168}, \frac{223}{168}\right\},$$
and
$$c^3+2c^2+c+1 = -\frac{5 \cdot 59 \cdot 223 \cdot 1222801}{2^{18} \cdot 3^6 \cdot 7^6}.$$ \smallskip

We now show that the map in this example shares a uniqueness property with the map $f_{-29/16}(x)$.

\newtheorem{lem}{Lemma}

\begin{lem}
(a) The only values of $m,n \in \mathbb{Z}$, for which $\textsf{N}_{K/\mathbb{Q}}(m+n\gamma^2) = \pm 1$ are $(m,n) = (\pm 1, 0), (0,\pm 1)$, $\pm (-1,1), \pm(2,-1), \pm(-7, 4)$. \smallskip

(b) The only values of $m,n \in \mathbb{Z}$, for which $\textsf{N}_{K/\mathbb{Q}}(m+n\gamma^2) = \pm 5$ are $(m,n) = \pm (1,1), \pm (-3,2)$. \smallskip

(c) The only values of $m,n \in \mathbb{Z}$, for which $\textsf{N}_{K/\mathbb{Q}}(m+n\gamma^2) = \pm 7$ are $(m,n) = \pm (-1,2), \pm (-5,3)$.
\label{lem:1}
\end{lem}

\begin{proof}
\noindent (a) We use the fact that the real root $\gamma$ of $x^3-x-1=0$ is a fundamental unit for $K = \mathbb{Q}(\gamma)$ (\cite[p. 519]{coh}).  Now $m+n \gamma^2$ is an algebraic integer, so it can have norm $\pm 1$ if and only if $m+n\gamma^2 = \varepsilon = \pm \gamma^k$ is a unit in $K$, where $k \in \mathbb{Z}$.  It is easy to see by induction that
$$\gamma^k =a_k + b_k \gamma +c_k \gamma^2, \ \ \textrm{where} \ b_k > 0 \ \textrm{for} \ k \ge 3.$$
This follows from $\gamma^3 = 1+\gamma$ and
\begin{align*}
\gamma^{k+1} & = a_k\gamma + b_k \gamma^2 +c_k \gamma^3\\
& = c_k + (a_k + c_k)\gamma + b_k \gamma^2\\
& = a_{k+1}+b_{k+1} \gamma + c_{k+1} \gamma^2;
\end{align*}
so that
$$a_{k+1} = c_k, \ \ b_{k+1} = a_k+c_k, \ \ c_{k+1} = b_k.$$
From this it follows easily that $b_{k+3} = b_{k+1} + b_k$, and then noting the beginning values $b_0 = 0, b_1 = 1, b_2 = 0$ yields the assertion.  Since a basis for $K/\mathbb{Q}$ is $\{1, \gamma, \gamma^2\}$, this shows that $m+n\gamma^2 \neq \pm \gamma^k$, for $k \ge 3$.  Hence, the only solutions of $m+n\gamma^2 = \varepsilon = \pm \gamma^k$ for non-negative $k$ come from $k = 0, 2$. \medskip

If $k = -l <0$, we let $\tilde b_{l} = b_{-l}$ for $ \ge 0$.  The recursion above implies that
$$\tilde b_l + \tilde b_{l-1} - \tilde b_{l-3} = 0, \ \ l \ge 0.$$
It follows that
$$\tilde b_l = \alpha_1 \gamma_1^{l}+\alpha_2 \gamma_2^{l} +\alpha_3 \gamma_3^l,$$
where $\gamma_1 = \gamma^{-1}$ and $\gamma_3 = \overline{\gamma}_2$ are the conjugates of $\gamma^{-1}$ and the roots of $x^3+x^2-1 = 0$.  The coefficients $\alpha_i$ are given by
\begin{align*}
\alpha_1 & = \frac{1}{(\gamma_1-\gamma_2)(\gamma_1-\gamma_3)},\\
\alpha_2 & = \frac{1}{(\gamma_2-\gamma_1)(\gamma_2-\gamma_3)},\\
\alpha_3 & = \frac{1}{(\gamma_3-\gamma_1)(\gamma_3-\gamma_2)}.
\end{align*}
We will now prove that $\tilde b_l = 0$ if and only if $l \in \mathcal{M} = \{-2, 0, 1, 5, 14\}$.  To do this we use Theorem 1 of the paper \cite[p. 359]{mt} of Mignotte and Tzanakis.  We choose $p = 59$, which splits completely in the field $K$ and in its normal closure $\Sigma = K(\sqrt{-23}) = \mathbb{Q}(\gamma_1, \gamma_2, \gamma_3)$.  We take a prime ideal $\mathfrak{p}$ of $\Sigma$ dividing $59$ for which
$$\gamma_1 \equiv 15, \ \gamma_2 \equiv 50, \ \gamma_3 \equiv 52 \ (\textrm{mod} \ \mathfrak{p}).$$
This amounts to taking an embedding of $\Sigma$ in the $p$-adic field $\mathbb{Q}_{59}$.  Then
$$\alpha_1 \equiv 39, \ \alpha_2 \equiv 16, \ \alpha_3 \equiv 4 \ (\textrm{mod} \ \mathfrak{p}).$$
We let $S=58$ in the Mignotte-Tzanakis theorem satisfying $\gamma_i^S \equiv 1 = A$ (mod $\mathfrak{p}$), for $i = 1,2,3$; and take
$$\mathcal{P} = \{-2, -1, 0, ..., 55\}$$
as a complete residue system modulo $S=58$.  We check the following conditions: \smallskip

\noindent (i) $\tilde b_m = 0$ for every $m \in \mathcal{M}$. \smallskip

\noindent (ii) If $n \in \mathcal{P}$ and $\tilde b_n \equiv 0$ (mod $59$), then $m \in \mathcal{M}$. \smallskip

\noindent (iii) $\tilde b_{m+S} = \tilde b_{m+58} \not \equiv \tilde b_m$ (mod $59^2$), for every $m \in \mathcal{M}$.  To check this we note that in $\mathbb{Q}_{59}$ we have
\begin{align*}
\gamma_1 & = 15 + 40 \cdot 59 + \cdots \equiv 2375 \ (\textrm{mod} \ 59^2),\\
\gamma_2 & = 50 + 57 \cdot 59 + \cdots \equiv 3413 \ (\textrm{mod} \ 59^2),\\
\gamma_3 & = 52 + 19 \cdot 59 + \cdots \equiv 1173 \ (\textrm{mod} \ 59^2);
\end{align*}
and
$$\alpha_1 \equiv 2871, \ \alpha_2 \equiv 2907, \ \alpha_3 \equiv 1184 \ (\textrm{mod} \ 59^2).$$
Finally, we check the following congruences for $\tilde b_{m+58}$ modulo $59^2$ in $\mathbb{Q}_{59}$, for $m \in \mathcal{M}$:
\begin{align*}
& \tilde b_{56} \equiv 1495; \ \ \tilde b_{58} \equiv 1121; \ \ \tilde b_{59} \equiv 767;\\
& \tilde b_{63} \equiv 354; \ \ \ \tilde b_{72} \equiv 3186.
\end{align*}
By the theorem of Mignotte-Tzanakis, we conclude that $\tilde b_l = 0$ if and only if $l \in \{-2, 0, 1, 5, 14\}$.  Hence, $b_k = 0$ if and only if $k \in \{0, 2, -1, -5, -14\}$, yielding the following solutions for $(m,n)$:
\begin{align*}
k=0: \ &\gamma^0 = 1 + 0\gamma^2,\\
k=2: \ &\gamma^2 = 0 + 1\gamma^2,\\
k = -1: \ &\gamma^{-1} = -1 + 1\gamma^2,\\
k = -5: \ &\gamma^{-5} = 2- 1\gamma^2,\\
k = -14: \ &\gamma^{-14} = -7+4\gamma^2.
\end{align*}
This proves (a). \medskip

\noindent (b) The associates of the element $2-\gamma$ are the only primes in $\textsf{R}_K$ with norm $5$.  Hence,
$$\textsf{N}_{K/\mathbb{Q}}(m+n\gamma^2) = \pm 5$$
if and only if
$$m+n\gamma^2 = \pm (2-\gamma) \gamma^k,$$
for some $k$.  As in part (a),
$$(2-\gamma) \gamma^k = a_k + b_k \gamma +c_k \gamma^2, \ \ a_k, b_k, c_k \in \mathbb{Z}.$$
The sequence $\{b_k\}$ satisfies the same recurrence as in (a), but with the starting values
$$b_{-1} = 0, \ b_0 = -1,\ b_1 = 2, \ b_2 = -1, \ b_3 = 1, \ b_4 = 1, \ b_5 = 0.$$
Since $b_3 = b_4 = 1$ are positive, and $b_{k+3} = b_{k+1} + b_k$, it follows that $b_k > 0$ for $k \ge 6$. \medskip

Now set $\tilde b_l = b_{-l}$. As in (a) we have
$$\tilde b_l = \alpha_1 \gamma_1^{l}+\alpha_2 \gamma_2^{l} +\alpha_3 \gamma_3^l,$$
with the same $\gamma_i \in \mathbb{Q}_{59}$ as before, and
$$\alpha_1 \equiv 158, \ \ \alpha_2 \equiv 3246, \ \ \alpha_3 \equiv 76 \ \ (\textrm{mod} \ 59^2).$$
With $\mathcal{M} = \{-5, 1\}, \mathcal{P} = \{-5, ..., -1, 0, ..., 52\}$ and $S = 58$, we check that conditions (i), (ii), (iii) of (a) hold,
where
$$\tilde b_{-5+S} = \tilde b_{53} \equiv 3009, \ \ \tilde b_{1+S} = \tilde b_{59} \equiv 413 \ (\textrm{mod} \ 59^2).$$
Then the Mignotte-Tzanakis theorem yields that $\tilde b_l = 0$ if and only if $l = -5,1$.  Hence $b_k = 0$ if and only if $k = 5, -1$, yielding the solutions
\begin{align*}
k = 5: \ &(2-\gamma)\gamma^5 = 1+\gamma^2;\\
k = -1:  \ &(2-\gamma) \gamma^{-1} = -3 + 2\gamma^2.
\end{align*}

\noindent (c) This is proved by the same method as in (a) and (b), using that the associates of $2+\gamma$ are the only primes in $\textsf{R}_K$ with norm $\pm 7$.  The solutions are
\begin{align*}
k = -1: \ & (2+\gamma)\gamma^{-2} = -1 + 2\gamma^2,\\
k = -9: \ & (2+\gamma)\gamma^{-9} = -5 + 3\gamma^2.
\end{align*}
The details are left to the reader.
\end{proof}

\begin{thm}
(a) The only values of $c = -\frac{A(m,n)}{B(m,n)}$, for which the rational cycle $\{x_1, x_2, x_3\}$ of $f_c(x) = x^2 + c$ contains an element $x_i = \frac{\pm 1}{4C(m,n)}$ with numerator $\pm 1$ are $c = -\frac{29}{16}, -\frac{29 \cdot 1709}{2^6 \cdot 3^2 \cdot 7^2}$. \smallskip

(b) The only values of $c = -\frac{A(m,n)}{B(m,n)}$, for which $\{x_1, x_2, x_3\}$ contains an element $x_i = \frac{\pm 5}{4C(m,n)}$ with numerator $\pm 5$ are $c = -\frac{29}{16}, -\frac{301}{2^4 \cdot 3^2}$. \smallskip

(c) The only values of $c = -\frac{A(m,n)}{B(m,n)}$, for which $\{x_1, x_2, x_3\}$ contains an element $x_i = \frac{\pm 7}{4C(m,n)}$ with numerator $\pm 7$ are $c = -\frac{29}{16}, -\frac{6469}{2^4 \cdot 3^2 \cdot 5^2}$.
\label{thm:7}
\end{thm}

\begin{proof}
By Lemma \ref{lem:1}(a), the only pairs $(m,n)$, for which the numerator of $x_1(m/n)$ is $\pm 1$, are $(m,n) = \pm(2,-1), \pm(-7,4)$, since the other pairs are not allowed.  The transformations $\beta(m,n) = (-n,m+n)$ and $\beta^2(m,n) = (-m-n,m)$, together with (\ref{eqn:10}), (\ref{eqn:14}) and (\ref{eqn:15}) show that the same values of $c$ will result from solving $\textsf{N}_{K/\mathbb{Q}}(m-n\gamma) = \pm 1$ and $\textsf{N}_{K/\mathbb{Q}}(m-n(\gamma-\gamma^2)) = \pm 1$.  This gives the only two possible values of $c$ stated in Part (a).  Parts (b) and (c) follow in the same way from Lemma \ref{lem:1}(b), (c).
\end{proof}

\begin{cor}
For any $c$, for which the map $f_c(x)$ has a rational $3$-cycle, the numerator of the rational number $N(c) = c^3+2c^2+c+1$ is divisible by at least three distinct primes.
\label{cor:2}
\end{cor}

\begin{proof}
For $c \neq  -\frac{29}{16}, -\frac{29 \cdot 1709}{2^6 \cdot 3^2 \cdot 7^2}$ this follows from Theorems \ref{thm:6} and \ref{thm:7}.  For these two values of $c$ it is clear:
$$c^3+2c^2+c+1 = -\frac{5 \cdot 7 \cdot 23}{2^{12}} \ \ \textrm{resp.} \ \ -\frac{5 \cdot 59 \cdot 223 \cdot 1222801}{2^{18} \cdot 3^6 \cdot 7^6}.$$
\end{proof}

This corollary shows that the rational $3$-cycle of $f_c(x)$ is $p$-adically attracting for at least three different odd primes $p$. \medskip

Theorem \ref{thm:7} shows that rational numbers with the three smallest possible numerators occur together in the $3$-cycle of $f_c(x)$ only for $c = -\frac{29}{16}$. \medskip

Similar arguments yield the following results.

\begin{thm}
(a) The only value of $c = -\frac{A(m,n)}{B(m,n)}$, for which the rational cycle $\{x_1, x_2, x_3\}$ of $f_c(x) = x^2 + c$ contains an element $x_i$ with any of the numerators $\pm 11, \pm 17$, or $\pm 25$ is $c = -\frac{421}{144}$. \smallskip

(b) The only values of $c = -\frac{A(m,n)}{B(m,n)}$, for which the rational cycle $\{x_1, x_2, x_3\}$ of $f_c(x) = x^2 + c$ contains an element $x_i = \frac{\pm 19}{4C(m,n)}$ with numerator $\pm 19$ are $c = -\frac{301}{144}$ and $c = -\frac{337 \cdot 673}{360^2}$.  \smallskip

(c) The only values of $c = -\frac{A(m,n)}{B(m,n)}$, for which the rational cycle $\{x_1, x_2, x_3\}$ of $f_c(x) = x^2 + c$ contains an element $x_i = \frac{\pm 23}{4C(m,n)}$ with numerator $\pm 23$ are $c = -\frac{301}{144}$ and $c = -\frac{43^2}{24^2}$.
\label{thm:8}
\end{thm}

\begin{proof}
(a) As in Theorem \ref{thm:7} it suffices to solve the norm equation $\textsf{N}(m+n\gamma^2) = a$ for a given numerator $a$.  We first solve the norm equations for the respective numerators $11, 17, 25$:
\begin{align}
\label{eqn:16} \textsf{N}(2\gamma-1) &= 11 \ \rightarrow \ 3 - \gamma^2 = (2\gamma-1) \gamma^{-1};\\
\label{eqn:17} \textsf{N}(3\gamma+2) & = 17 \ \rightarrow \ 1+2\gamma^2 =  (3\gamma+2) \gamma^{-1};\\
\notag \textsf{N}(\gamma + 3) & = 25, \ \gamma+3 = (2\gamma+1)^2(\gamma-1)\\
\label{eqn:18} & \ \ \ \ \ \ \ \ \rightarrow -2+3\gamma^2 = \pm (\gamma+3) \gamma^{-1};\\
\notag \textsf{N}(\gamma^2+2\gamma+3) &= 25, \ \tilde \pi = \gamma^2+2\gamma+3 \ \textrm{prime}\\
 \label{eqn:19} & \ \ \ \ \ \ \ \ \rightarrow m+n\gamma^2 = \pm (\gamma^2+2\gamma+3) \gamma^k \ \textrm{no solution}.
\end{align}

 \begin{table}
  \centering 
  \caption{Data for the proof of Theorem \ref{thm:8}.}

\noindent \begin{tabular}{|c|cl||c|c|c|c|c|c|c|}
\hline
& & & & & & & & & \\
Num	& $\pi$ & $p$ & $\gamma_1$ & $\gamma_2$ & $\gamma_3$  & $a_1$ & $a_2$ & $a_3$ & $\tilde b_{m+S}$\\
\hline
11 & $2\gamma-1$ & 101 & 1409 & 4507 & 4284 & 2768 & 4450 & 2985 & 6060\\
17 & $3\gamma+2$ & 173 & 23690 & 21569 & 14598 & 9587 & 1499 & 18846 & 28026\\
25 & $\gamma+3$ & 59 & 2375 & 3413 & 1173 & 273 & 846 & 2363 & 3422\\
25 & $\tilde \pi$ & 101 & 1409 & 4507 & 4284 & 2162 & 8668 & 9574 & --\\
19 & $\gamma^2+2$ & 59 & 2375 & 3413 & 1173 & 997 & 2500 & 3465 & 649\\
23 & $\gamma-3$ & 59 & 2375 & 3413 & 1173 & 452 & 809 & 2221 & 2301\\
23 & $2\gamma+3$ & 59 & 2375 & 3413 & 1173 & 2376 & 3414 & 1174 & 1062\\
 \hline
\end{tabular}
\label{tab:1}
\end{table}

In the two cases for the numerator $25$, note that
$$5 = \varepsilon (2\gamma+1)(\gamma^2+2\gamma+3), \ \varepsilon = 2\gamma^2-\gamma-2 \in \textsf{R}_K^\times,$$
is a product (in the ring of integers $\textsf{R}_K$ of $K$) of the primes $2\gamma+1$ and $\gamma^2+2\gamma+3$ having norms $5$ and $25$, respectively, corresponding to the factorization
$$x^3-x-1 \equiv 3(2x + 1)(x^2 + 2x + 3) \ (\textrm{mod} \ 5).$$
Using the method of Lemma \ref{lem:1}, we show that the solutions in (\ref{eqn:16})-(\ref{eqn:18}) are unique, while (\ref{eqn:19}) has no solution.  The calculations are displayed in Table \ref{tab:1}.  For each numerator we find a prime $p$ which splits in $K$, and find $p$-adic approximations of the roots $\gamma_i$ (mod $p^2$) of $x^3+x^2-1=0$.  As in Lemma \ref{lem:1}, the linear recurring sequence $\tilde b_k = b_{-k}$ is defined by
$$\pi \gamma^k = a_k +b_k \gamma +c_k \gamma^2, \ \ k \in \mathbb{Z},$$
where $\pi= 2\gamma-1, 3\gamma+2, \gamma+3$, respectively for (\ref{eqn:16})-(\ref{eqn:18}) and $\pi = \tilde \pi = \gamma^2+2\gamma+3$ for (\ref{eqn:19}).  In the Mignotte-Tzanakis theorem we take $S = p-1$ in each case and solve for the coefficients $a_i$ modulo $p^2$ in the representation
$$\tilde b_k = \sum_{i=1}^3{a_i \gamma_i^k}.$$
The prime $p$ is chosen so that $\tilde b_k  \equiv 0$ (mod $p$) if and only if $k \equiv 1$ (mod $p-1$) in (\ref{eqn:16})-(\ref{eqn:18}) , while $\tilde b_k \not \equiv 0$ (mod $p$) for all $k$ (mod $101$) in (\ref{eqn:19}).  The final column in the table gives the value $\tilde b_{1+S} = \tilde b_p$ modulo $p^2$, showing that $\tilde b_p \not \equiv \tilde b_1 \equiv 0$ (mod $p^2$).  The Mignotte-Tzanakis theorem implies that the solutions in (\ref{eqn:16})-(\ref{eqn:18})  are unique, while it is clear that (\ref{eqn:19}) has no solution, since $\tilde b_k$ is never $0$ (mod $101$). \smallskip

Plugging the solutions $(m,n) = (3,-1), (1,2), (-2,3)$ from (\ref{eqn:16})-(\ref{eqn:18}) into $y_1(s)$ yields $c = -421/144$, completing the proof of (a). \smallskip

Parts (b) and (c) are proved the same way, using the data for $19$ and $23$ in Table \ref{tab:1}.  For the numerator $19$ there are two solutions
$$m + n\gamma^2 = 2+\gamma^2, \ \ m + n\gamma^2 = (2+\gamma^2)\gamma^{-10} = 9-5\gamma^2,$$
and with $S=58$ we have $\tilde b_{0+58} \equiv \tilde b_{10+58} \equiv 649$ (mod $59^2$).  For the numerator $23$ there is one solution for each of the primes $\pi_1 = \gamma-3, \pi_2 = 2\gamma+3$ dividing $23$:
$$m+n\gamma^2 = (\gamma-3)\gamma^{-1} = 4-3\gamma^2, \ \ m+n\gamma^2 = (2\gamma+3)\gamma^{-1} = -1+3\gamma^2.$$
\end{proof}
 
 \section{The numerator graph $\Gamma$.}
 
 \noindent {\bf Definition.} We define a graph on the absolute values of possible numerators of the $3$-periodic points $x_i(m/n)$ as follows.  If $a,b > 0$ are two such numerators, then the graph $\Gamma$ contains an edge $(a,b)$ if and only if there is a map $f_c$ for which two of the numerators in the rational $3$-cycle $\{x_1, x_2, x_3\}$ are $\pm a$ and $\pm b$.  \medskip
 
Define the polynomials $t_i(m,n)$:
\begin{align}
\label{eqn:20} t_1(m,n) &= m^3+2m^2n+mn^2+n^3,\\
\label{eqn:21} t_2(m,n) &= m^3-mn^2-n^3,\\
\label{eqn:22} t_3(m,n) &= m^3+2m^2n+3mn^2+n^3.
\end{align}
These are the numerators of the respective elements of the $3$-cycle $\{x_1, x_2, x_3\}$ (see Theorem \ref{thm:5} and equations (\ref{eqn:8}), (\ref{eqn:11}), (\ref{eqn:12})) of the map $f_c$, where $c = y_1(m/n)$.  It is straightforward to verify that
\begin{align*}
t_1(\beta(m,n)) &= t_3(m,n),\\
t_2(\beta(m,n)) &= -t_1(m,n),\\
t_3(\beta(m,n)) &= t_2(m,n),
 \end{align*}
and $t_i(\beta^3(m,n)) = t_i(-m,-n) = -t_i(m,n)$.  Using these polynomials, the definition of $\Gamma$ can be restated as follows.  Two positive integers $a$ and $b$, each representable as the norm of a primitive element in the module $M_{\gamma} = \mathbb{Z}[1, \gamma] \subset K = \mathbb{Q}(\gamma)$, are connected by an edge if and only if there is a pair of integers $(m,n)$ satisfying
\begin{align*}
a = |t_i(m,n)|, & \ b = |t_j(m,n)|, \ \textrm{where} \ i, j \in \{1,2,3\}, \ i \neq j,\\
 & \ mn(m+n) \neq 0, \ \gcd(m,n) = 1.
 \end{align*}
To calculate $\Gamma$, we start with a positive integer $a_1$ and find all (allowable) solutions of the Thue equation $t_1(m,n) = a_1$.  For each pair $(m,n)$ we compute $t_2(m,n)$ and $t_3(m,n)$.  This is the same as computing $t_1(\beta^2(m,n)) =  t_1(-(m+n),m)$ and $t_1(\beta(m,n)) = t_1(-n, m+n)$, by the above formulas.  This gives us the triangular subgraph $\{a_1, a_2, a_3\} \subset \Gamma$, where $a_i = |t_i(m,n)|$.  Now find all allowable solutions of the equations $t_2(m,n) = a_2$ and $t_3(m,n) = a_3$ and continue.  \medskip 
 
Theorem \ref{thm:6} shows that $\Gamma$ does not contain any edge of the form $(a,a)$.  According to the above definition, the subgraph whose nodes are $\{11, 17, 25\}$ is a closed subgraph, meaning that it is the connected component of any of its elements.  This is because these numerators only occur for $c = -\frac{421}{144}$, by Theorem \ref{thm:8}(a).  Theorem \ref{thm:7}(b) and Theorem \ref{thm:8}(b) and (c) (see the examples following Theorem \ref{thm:5}) show that $\Gamma$ contains the edges
 \begin{align*}
 (5, 19), (5, 23), & (19, 23), (19, 469), (469, 629), (19, 629),\\
 & (23, 49), (49, 55), (23, 55).
  \end{align*}
These edges form three triangles, each corresponding to a specific value of $c$.  Thus the nodes $19$ and $23$ are each $2$-step connected to the node $1$, while $55, 49$ are each $3$-step connected to $1$.  Working on Pari, we have verified that the adjoining Figure \ref{fig:2} represents the connected component of the node $1$. \medskip

\begin{figure}[h!]
	\centering
        \includegraphics[scale=.55]{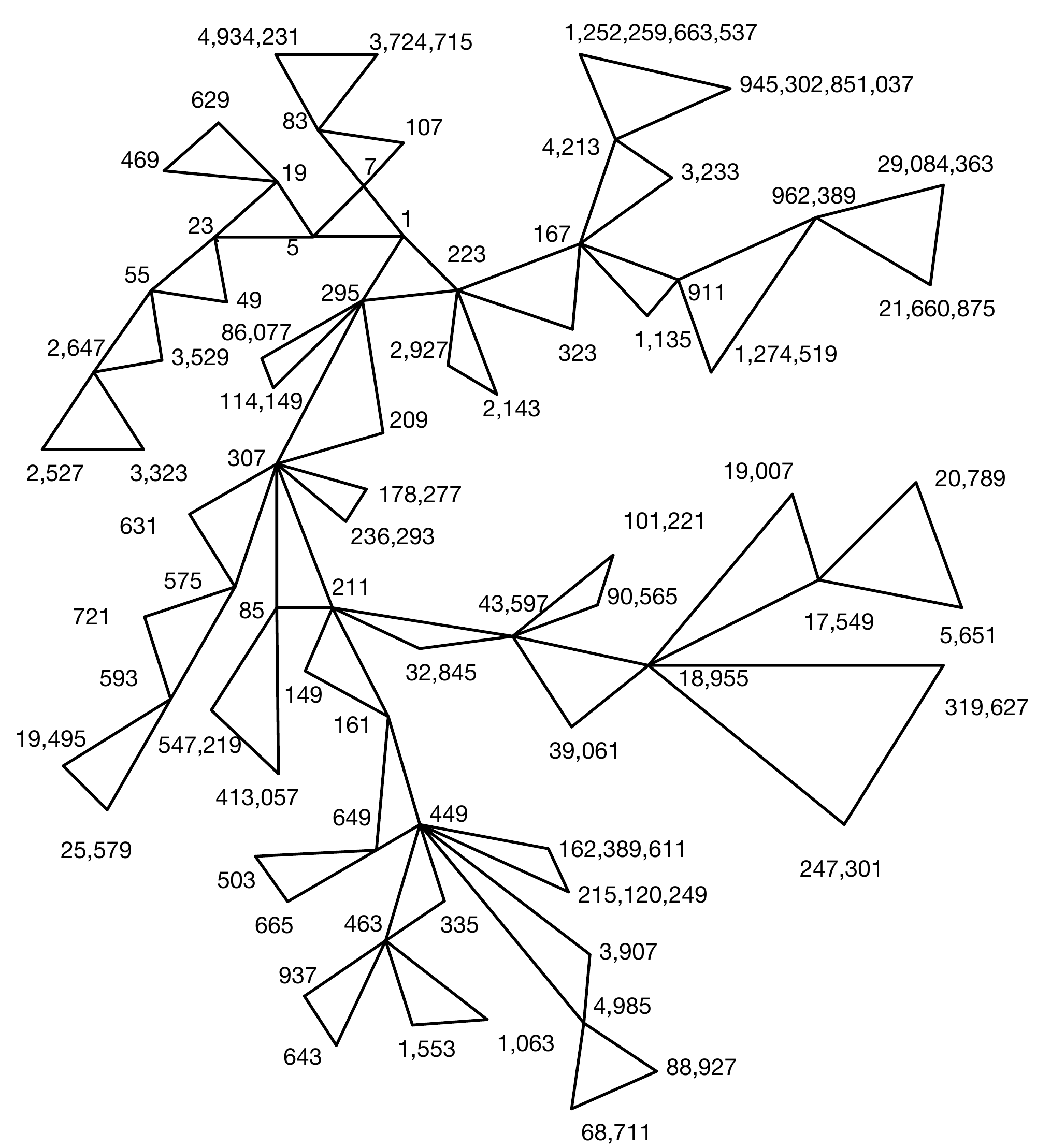}
        \caption{The Cholla: Connected component of $\{1\}$ in the graph $\Gamma$.}
\label{fig:2}
\end{figure}

The nodes $307$ and $449$ in Figure \ref{fig:2} are especially interesting, and give some hints about the relationship of the triangles in the graph $\Gamma$ to the arithmetic in the ring of integers $\textsf{R}_K$ of the field $K = \mathbb{Q}(\gamma)$.  Calculating on Pari, we find the solutions
$$(m,n) = (-9, 7), (-4,7), (-1, 7), (65,-37)$$
of the Thue equation
$$t_1(m,n) = m^3+2m^2n+mn^2+n^3 = 307.$$
These solutions correspond to the following integers in $\textsf{R}_K$ having norm $307$.  The three integers
$$\alpha_1 = -9+7\gamma^2, \ \alpha_2 = -4+7\gamma^2, \ \alpha_3=-1+7\gamma^2$$
represent distinct (non-associate) primes in $\textsf{R}_K$, while
$$65-37\gamma^2 = \gamma^{-17}(-4+7\gamma^2)$$
is an associate of $\alpha_2$.  Thus, either one or two solutions of the Thue equation $t_1(m,n) = 307$ correspond to each prime divisor of $307$ in $\textsf{R}_K$.  Each of the above solutions corresponds to a different value of $c$:
\begin{align}
\notag y_1\left(-\frac{9}{7}\right) &= c_1 = -\frac{463 \cdot 547}{2^4 \cdot 3^4 \cdot 7^2},\\
\notag y_1\left(-\frac{4}{7}\right) &= c_2 = -\frac{127 \cdot 463}{2^6 \cdot 3^2 \cdot 7^2},\\
\label{eqn:23} y_1\left(-\frac{1}{7}\right) &= c_3 = -\frac{127 \cdot 547}{2^4 \cdot 3^2 \cdot 7^2},\\
\notag y_1\left(-\frac{65}{37}\right) &= c_4 = -\frac{757 \cdot 42039677}{2^6 \cdot 5^2 \cdot 7^2 \cdot 13^2 \cdot 37^2}.
\end{align}
(Note the prime factors of the first three values.)  These $c$-values correspond in turn to the four triangles in $\Gamma$ which meet at the node $307$:
\begin{align*}
c_1 &\longrightarrow \{307, 575, 631\};\\
c_2 &\longrightarrow \{85, 211, 307\};\\
c_3 &\longrightarrow \{209, 295, 307\};\\
c_4 &\longrightarrow \{307, 178277,236293\}.
\end{align*}

A similar situation also exists for the prime $449$.  The solutions of $t_1(m,n) = 449$ are
$$(m,n) = (7, 1), (4,5), (-18, 11), (-630,359);$$
where
$$\alpha_1 = 7+\gamma^2, \ \alpha_2 = 4+5\gamma^2, \ \alpha_3 = -18+11\gamma^2$$
are three distinct prime divisors of $449$ in $R_K$, while
$$-630+359\gamma^2 = \gamma^{-32} (7+\gamma^2)$$
is an associate of $\alpha_1$.  Again, one or two solutions of the Thue equation correspond to each prime divisor of $449$ in $\textsf{R}_K$.  (See Conjecture \ref{conj:3} below.) \medskip

Calculating on Pari, we have also verified the following triangular closed subgraphs:
\begin{align*}
\{53, 347, 437\}, \ \{79, &89, 121\}, \ \{115, 4483, 5891\},\\
\{119, 181, 191\}, \ \{157, &17497, 23243\}, \ \{187, 1477, 1877\},\\
\{199, 320149, 424189\}, \ \{229, &11471, 15101\}, \ \{2809, 3353, 4705\}.
\end{align*}

\begin{figure}[h!]
	\centering
        \includegraphics[scale=.55]{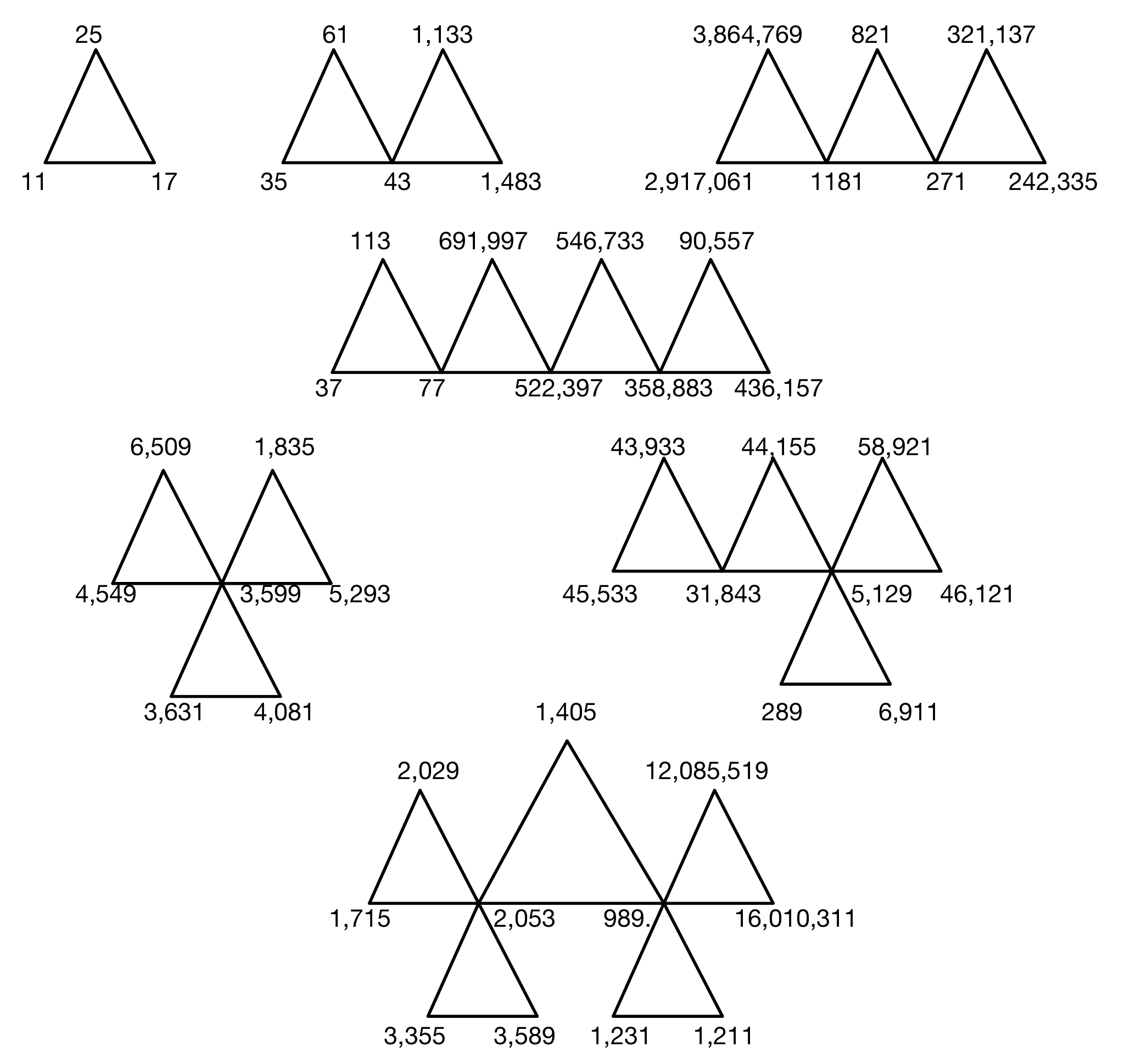}
        \caption{Some small connected components of $\Gamma$.}
\label{fig:3}
\end{figure}

There are also other connected components of numbers $a \le 500$, such as ``bowties'' connecting two triangles, as in
$$\{\{35,43,61\},\{43, 1133, 1483\}\} \ \textrm{and} \ \{\{59,101,109\},\{59, 851, 1151\}\};$$
and the connected chains of three and four triangles pictured in Figure \ref{fig:3}.
We have also found a ``three-leaf clover'' of three triangles connecting at the single node $3599$:
$$\{\{3599, 1835, 5293\}, \{3599, 3631, 4081\}, \{3599, 4549, 6509\}\};$$
and two connected components containing five triangles connected to the nodes $28891$ and $2341711$.  Note that $3599 = 59 \cdot 61$, where $59$ splits in the field $K = \mathbb{Q}(\gamma)$; as do the prime factors of $28891=167 \cdot 173$ and $2341711 = 271 \cdot 8641$.  There is also the connected component of the node $883$, which contains 10 triangles and 21 nodes, including the two nodes $883$ and $1451$ where three triangles meet.  See Figure \ref{fig:4}.  Notice that both primes $883, 1451$ split completely in $K$ as well.  So far, we have not found another connected component of $\Gamma$ which is as complicated and interesting as the connected component of $1$.  \medskip

\begin{figure}[h!]
	\centering
        \includegraphics[scale=.55]{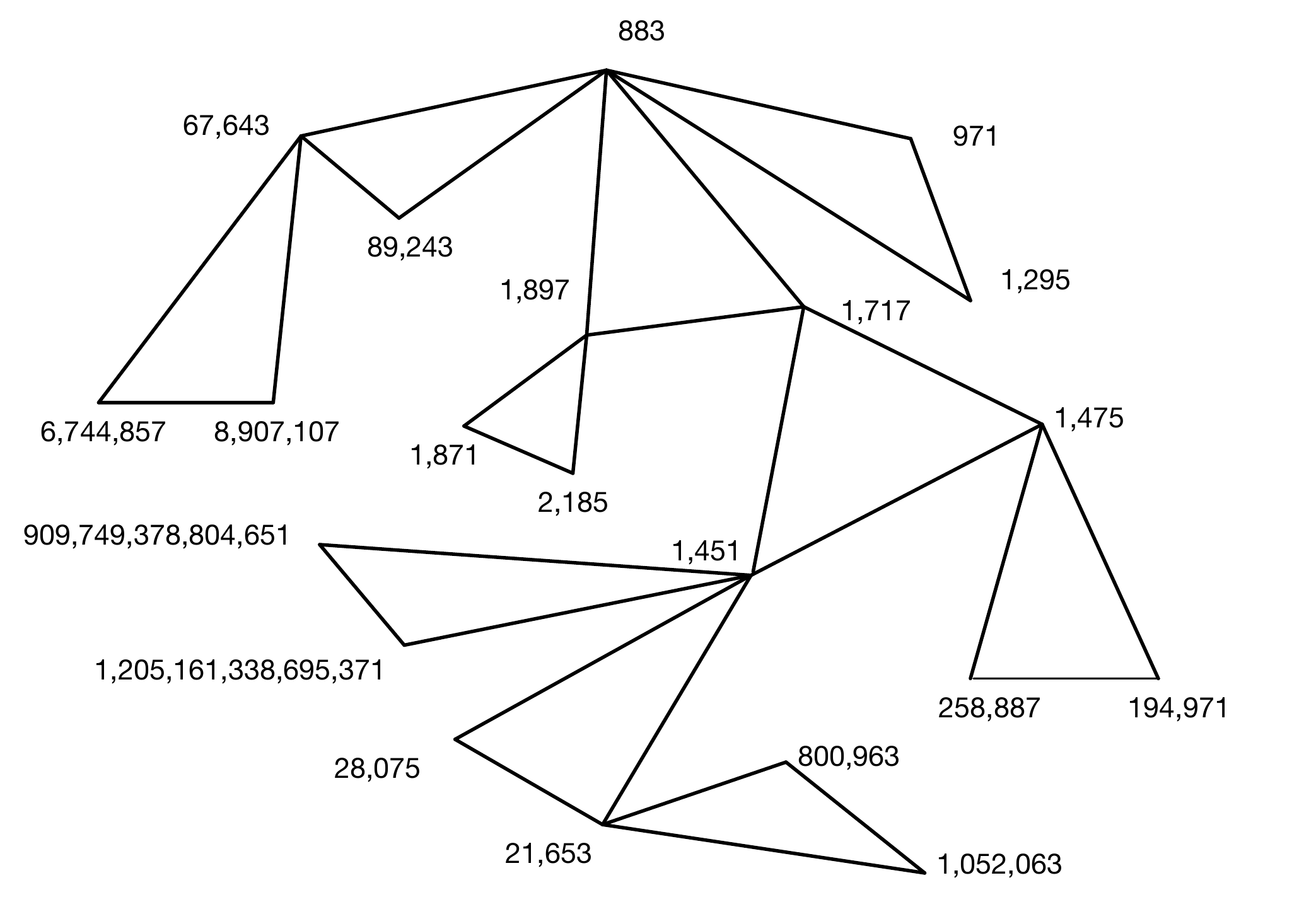}
        \caption{Connected component of $\{883\}$.}
\label{fig:4}
\end{figure}

Assume that $\{a_1, a_2, a_3\}$ are the absolute values of the numerators in a rational $3$-cycle $\{x_1, x_2, x_3\}$ of some map $f_c(x)$.  We say the triangle $\{a_1, a_2, a_3\}$ corresponds to the map $f_c$ and the value $c$.

\begin{thm}
If a given triangle $\{a_1, a_2, a_3\}$ in the graph $\Gamma$ corresponds to the map $f_c$ with parameter $c$, then $c$ is unique.
In other words, the absolute values of the numerators of the elements of the orbit $\{x_1, x_2, x_3\}$ of $f_c$ determine $c$.
\label{thm:9}
\end{thm}

\begin{proof}
We make use of the map $\beta(m,n) = (-n, m+n)$ defined on pairs $(m,n)$ with $mn(m+n) \neq 0$ and $\gcd(m,n) = 1$.  It is easily checked that among the pairs $\beta^i(m,n)$, for $0 \le i \le 5$, there is exactly one pair with positive coordinates.  In terms of the polynomials $t_i$ from (\ref{eqn:20})-(\ref{eqn:22}) we have
\begin{align*}
x_1\left(\frac{m}{n}\right) &= \frac{t_1(m,n)}{2mn(m+n)},\\
x_2\left(\frac{m}{n}\right) &= \frac{t_2(m,n)}{2mn(m+n)},\\
x_3\left(\frac{m}{n}\right) &= -\frac{t_3(m,n)}{2mn(m+n)}.
\end{align*}
We will also make use of the following chain of equations:
\begin{equation}
\label{eqn:34} \frac{f_c^2(x_3)-f_c^2(x_1)}{f_c(x_3)-f_c(x_1)} = -\frac{f_c(x_1)-f_c(x_2)}{f_c(x_1)-f_c(x_3)} = -\frac{x_2-x_3}{x_2-x_1} = s,
\end{equation}
in which the last equality follows from the equations (\ref{eqn:8}), (\ref{eqn:11}), (\ref{eqn:12}) and
$$x_3(s)-x_2(s) = -\frac{s^2 + s + 1}{s + 1}, \ \ x_2(s)-x_1(s) = -\frac{s^2 + s + 1}{s(s + 1)}.$$
The first expression in (\ref{eqn:34}) is a dynamical unit, in the language of \cite{ms1}.  In our case it is a unit in the ring of rational numbers whose denominators are divisible at most by prime factors of the quantity $mn(m+n)$.  The second expression in this formula is also given by
$$-\frac{f_c(x_1)-f_c(x_2)}{f_c(x_1)-f_c(x_3)} = -\frac{x_1(s)^2-x_2(s)^2}{x_1(s)^2-x_3(s)^2} = -\frac{t_1(m,n)^2-t_2(m,n)^2}{t_1(m,n)^2-t_3(m,n)^2},$$
since the $x_i$ have the same denominators.  Now putting this together with (\ref{eqn:34}) and applying $\beta^2(m,n) = (-(m+n),m)$ to the polynomials $t_i$ yields $\beta^2(t_1, t_2, -t_3) = (t_2,-t_3, t_1)$ and the following formulas:
\begin{align}
\notag -\frac{t_1(m,n)^2-t_2(m,n)^2}{t_1(m,n)^2-t_3(m,n)^2} &= \frac{m}{n} = s,\\
\label{eqn:35} -\frac{t_2(m,n)^2-t_3(m,n)^2}{t_2(m,n)^2-t_1(m,n)^2} &= -\frac{m+n}{m} = -\frac{s+1}{s} = \psi(s),\\
\notag -\frac{t_3(m,n)^2-t_1(m,n)^2}{t_3(m,n)^2-t_2(m,n)^2} &= -\frac{n}{m+n} = -\frac{1}{s+1} = \psi^2(s).
\end{align}
Assume now that there are two different values $c_1 = y_1\left(\frac{m_1}{n_1}\right), c_2 = y_1\left(\frac{m_2}{n_2}\right)$ corresponding to the triangle $\{a_1, a_2, a_3\}$.  By (\ref{eqn:10}) and the above remarks we may apply the map $\beta$ separately to $(m_1,n_1)$ and $(m_2,n_2)$ to arrange that $m_i, n_i > 0$ for $i = 1, 2$.  We write $x_j^{(i)}$ for the elements of the $3$-cycle of $f_{c_i}$.  Then the above formulas give that
$$x_1^{(i)} > 0, \ \ x_3^{(i)} < 0, \ \ i = 1,2.$$
We can assume therefore that
$$a_1 = t_1(m_1,n_1), \ a_2 = |t_2(m_1,n_1)|, \ a_3 = t_3(m_1,n_1).$$
There are essentially two cases.  In the first case the numerators $a_1, a_2, a_3$ appear in the $3$-cycle $\{x_1^{(2)}, x_2^{(2)}, x_3^{(2)} \}$ of $f_{c_2}$ in a cyclic (even) permutation of their appearance in the $3$-cycle $\{x_1^{(1)}, x_2^{(1)}, x_3^{(1)} \}$ of $f_{c_1}$.  Since the left sides of the formulas in (\ref{eqn:35}) are related by the cyclic permutation $(1, 2, 3)$, it follows that for some $i$,
$$s_1 = -\frac{a_1^2-a_2^2}{a_1^2-a_3^2} = \psi^i(s_2).$$
Now $y_1(s)$ is invariant under $\psi(s)$, so this shows that $c_1 = y_1(s_1) = y_1(\psi^i(s_2)) = c_2$. \medskip

In the second case, the appearance of the numerators $a_1, a_2, a_3$ in the $3$-cycle of $f_{c_2}$ is an odd permutation of their appearance in the $3$-cycle for $f_{c_1}$.  It follows that switching $a_2$ and $a_3$ in the last formula gives
$$\frac{1}{s_1} = -\frac{a_1^2-a_3^2}{a_1^2-a_2^2} = \psi^i(s_2).$$
Hence, $c_1 = y_1(s_1) = y_1\left(\frac{m}{n}\right)$ and
$$c_2 = y_1(\psi^i(s_2)) = y_1\left(\frac{1}{s_1}\right) = y_1\left(\frac{n}{m}\right).$$
This implies that $t_1(m,n) = t_i(n,m)$ for $i = 1$ or $3$ or $t_1(m,n) = \pm t_2(n,m)$.  However, we have the formulas
\begin{align*}
t_1(m,n) - t_1(n,m) & = mn(m - n),\\
t_1(m,n) - t_2(n,m) & = m(m + n)(2m + n),\\
t_1(m,n) + t_2(n,m) & = n(m^2 + mn + 2n^2),\\
t_1(m,n) - t_3(n,m) & = -mn(m + n);
\end{align*}
and only the first of these expressions can equal $0$ for positive $m$ and $n$.  In that case $m = n = 1$ and $c_1 = c_2$.  This completes the proof.
\end{proof}

\begin{cor} With the polynomials $t_i = t_i(m,n)$ defined as in (\ref{eqn:20})-(\ref{eqn:22}) and $s = \frac{m}{n}$ we have
$$s = \frac{t_2+t_3}{t_1-t_2}.$$
If $\{a_1, a_2, a_3\}$ are the absolute values of the $t_i$, then 
$$s = -\frac{a_1^2-a_2^2}{a_1^2 - a_3^2}.$$
\label{cor:3}
\end{cor}
\begin{proof}
This is immediate from the first equation in $(\ref{eqn:35})$ and the last equality in (\ref{eqn:34}), since the elements $x_i$ have equal denominators.
\end{proof} 

\begin{conj}
i) Every connected component of the graph $\Gamma$ is a finite graph. \smallskip

ii) Two triangles in $\Gamma$ never share an edge. \smallskip

iii) There are infinitely many connected components of $\Gamma$ consisting of single triangles $\{a_1, a_2, a_3\}$. \smallskip

iv) At any node $a$ where three or more triangles connect, either $23 \mid a$ or $a$ is divisible by at least one prime $p$ which splits in $K = \mathbb{Q}(\gamma)$. \smallskip

v) For any positive vertex $a$ in $\Gamma$, the number of distinct triangles meeting at $a$ equals the number of distinct (allowable) solutions of the Thue equation
$$m^3+2m^2n+mn^2+n^3 = a, \ \ mn(m+n) \neq 0 \ \ (m,n) = 1.$$
\label{conj:2}
\end{conj}

\noindent {\bf Remark.} It would follow from Conjecture \ref{conj:2}(ii) that every triangle in $\Gamma$ corresponds to a value of $c$.  The alternative is that three nodes $\{a_1, a_2, a_3\}$ are numerators of $3$-periodic points of different maps $f_c$ two at a time.  In other words, there would exist at least one positive value $a'$ for which $\{a_1, a_2, a'\}$ corresponds to $f_{c'}$, where $a' \neq a_3$; and the same for the pairs $a_1, a_3$ and $a_2, a_3$; and that the three corresponding values of $c$ are distinct.  But then there are two triangles in $\Gamma$ sharing the edge $(a_1, a_2)$. \medskip

Conjecture \ref{conj:2}(iv) suggests the following conjecture related to the arguments in Lemma \ref{lem:1} and Theorem \ref{thm:8}.  See Theorem \ref{thm:12} below.

\begin{conj}
For any pair $(m,n)$ with $mn(m+n) \neq 0$ and $\gcd(m,n) = 1$, for which $\textsf{N}_{K/\mathbb{Q}}(m+n\gamma^2) \neq \pm 1$, there is at most one integer $k \neq 0$ for which $\gamma^k (m+n\gamma^2)$ lies in the $\mathbb{Z}$-module $M = \mathbb{Z}[1, \gamma^2]$.
\label{conj:3}
\end{conj}

We will call a triangle that corresponds to a map $f_c$ a {\it $c$-triangle}.  The following theorem shows that Conjecture \ref{conj:2}(v) holds for $c$-triangles in place of triangles.

\begin{thm}
The number of $c$-triangles connected to a vertex $a$ in $\Gamma$ equals the number of distinct allowable solutions of the equation $t_1(m,n) = m^3+2m^2n+mn^2+n^3 = a$.
\label{thm:10}
\end{thm}

\begin{proof} By Corollary \ref{cor:3} (or the proof of Theorem \ref{thm:9}) we conclude that every $c$-triangle uniquely determines $c$ and a corresponding orbit of $(m,n)$ under the group of substitutions generated by $\beta$.  If $a$ is a vertex in the triangle corresponding to $c$, then by applying a power of $\beta$ we will have $a = t_1(m,n)$ for some pair $(m,n)$ in this orbit.  Hence each $c$-triangle containing $a$ corresponds to at least one solution of this equation.  Suppose that $t_1(m,n) = a$ and that one of the other pairs $\beta^i(m,n)$ is also a solution, i.e., that $t_1(\beta^i(m,n)) = a$.  Then we have, depending on the value of $i$:
\begin{align*}
t_1(m,n) - t_1(\beta(m,n)) &= t_1(m,n) - t_1(-n,m+n)\\
& = t_1(m,n) - t_3(m,n) = -2mn^2;\\
t_1(m,n) - t_1(\beta^2(m,n)) &= t_1(m,n) - t_1(-m-n,m)\\
& = t_1(m,n) - t_2(m,n) = 2n(m^2+mn+n^2);\\
t_1(m,n) - t_1(\beta^3(m,n)) &= t_1(m,n) - t_1(-m,-n) = 2t_1(m,n);\\
t_1(m,n) - t_1(\beta^4(m,n)) &= t_1(m,n) - t_1(n,-m-n)\\
& = t_1(m,n) + t_3(m,n) = 2(m + n)(m^2 + mn + n^2);\\
t_1(m,n) - t_1(\beta^5(m,n)) &= t_1(m,n) - t_1(m+n,-m)\\
& = t_1(m,n) + t_2(m,n) = 2m^2(m + n).
\end{align*}
Since $mn(m+n) \neq 0$, it follows that none of these expressions can equal $0$.  Hence, only the pair $(m,n)$ is a solution of $t_1(m,n) = a$.  This shows that there is a $1-1$ correspondence between $c$-triangles with vertex $a$ and solutions of $t_1(m,n) = a$.
\end{proof} 

We will now prove the following result related to Conjecture \ref{conj:2}(ii).

\begin{thm}
There are at most a finite number of $c$-triangles which share an edge with another $c$-triangle in the graph $\Gamma$.
\label{thm:11}
\end{thm}

\begin{proof}  Assume that two $c$-triangles
$$T_1 = \{a_1^{(1)}, a_2^{(1)}, a_3^{(1)}\} \ \textrm{and} \ T_2 = \{a_1^{(2)}, a_2^{(2)}, a_3^{(2)}\}$$
in $\Gamma$ share an edge, where $a_i^{(1)} = |t_i(m,n)|$ and $a_i^{(2)} = |t_i(x,y)|$.  Then let the vertices on the common edge be $a,b$.  By applying powers of the map $\beta$ separately to these triangles, cyclically permuting the representations of the $a_i^{(j)}$ in terms of $t_i$ and renaming the $a_i^{(j)}$, we can arrange that the vertices which do not lie on the common edge are $a_3^{(1)} = |t_3(m,n)|$ and $a_3^{(2)} = |t_3(x,y)|$.  Further, since $a = |t_1(x,y)| = |t_i(m,n)|$ for $i = 1$ or $2$, we have
$$t_1(x,y) = \pm t_i(m,n), \ \ t_2(x,y) = \pm t_j(m,n), \ \ \{i,j\} = \{1, 2\}.$$
Applying $\beta^3$ to $T_1$, if necessary, which replaces $(m,n)$ by $(-m,-n)$ and leaves $i$ and $j$ fixed, we can assume that
$$t_1(x,y) = t_i(m,n), \ \ t_2(x,y) = \pm t_j(m,n), \ \ \{i,j\} = \{1, 2\}.$$
Thus, there are four cases, according as $i = 1$ or $2$ and the sign in the second equation is plus or minus. \medskip

{\it Case 1.}
Assume that
\begin{equation}
\label{eqn:36} t_1(x,y) = t_1(m,n) \ \ \textrm{and} \ \  t_2(x,y)=t_2(m,n).
\end{equation}
If these equations have a common solution, then the following resultant must be zero:
\begin{equation}
\label{eqn:37} \textrm{Res}_m(t_1(x,y) - t_1(m,n), t_2(x,y) - t_2(m,n)) \ = 8(y - n)(n^2 + ny + y^2) F(x,y),
\end{equation}
where
\begin{align*}
F(x,y) = \ & n^6 - n^3x^3 + x^6 + (-4n^3x^2 + 3x^5)y + (-3n^3x + 6x^4)y^2\\
& \ + (-2n^3 + 7x^3)y^3 + 6x^2y^4 + 3xy^5 + y^6.
\end{align*}
The polynomial $F$ is homogeneous of degree $6$ in $(x,y,n)$.  Putting $(x,y) = (nx, ny)$ and dividing by $n^6$ amounts to setting $n=1$, so we only need to find the solutions of the curve $F_1(x,y)= 0$, with
\begin{align*}
F_1(x,y) = \ & y^6 + 3xy^5 + 6x^2y^4 + (7x^3 - 2)y^3 + (6x^4 - 3x)y^2\\
& \ + (3x^5 - 4x^2)y + x^6 - x^3 + 1.
\end{align*}
The curve $F_1(x,y) = 0$ is birationally equivalent to the elliptic curve
\begin{equation}
\label{eqn:38} E: w^2 = z^3 - \frac{28}{3}z -\frac{1261}{108}, \ \ j(E) = -\frac{2^{18} \cdot 7^3}{19^3},
\end{equation}
by the mapping $(z,w) \rightarrow (x,y)$, where
\begin{align*}
x & = \frac{-(3z - 16)(9z^2 + 39z + 49)(54z^3 + 216z^2  + 783z + 982 - 288w + 54wz )}{2(729z^6 + 5103z^5 + 40581z^4 + 97146z^3 + 195264z^2 + 487197z + 483193)},\\
y & = \frac{(27z^3 + 216z^2 + 9z - 541)(54z^3 + 216z^2 + 783z + 982 - 288w + 54wz)}{2(729z^6 + 5103z^5 + 40581z^4 + 97146z^3 + 195264z^2 + 487197z + 483193)};
\end{align*}
and the denominator in these expressions is irreducible.  Furthermore, the inverse mapping has the form
$$z = \frac{u(x,y)}{3x(x+1)^2}, \ \ w = \frac{v(x,y)}{2x(x+1)^3},$$
where $u, v$ are polynomials in $x$ and $y$.  Since
\begin{align*}
F_1(0,y) &= (y-1)^2(y^2 + y + 1)^2,\\
F_1(-1,y) & = (y - 1)(y^5 - 2y^4 + 4y^3 - 5y^2 + 4y - 3),
\end{align*}
$(x,y) = (0,1), (-1,1)$ are the only points where the rational expressions for $z$ and $w$ are undefined. \medskip

Now the elliptic curve $E$ has rank $0$ and torsion group $T = \{O, (\frac{16}{3}, \pm \frac{19}{2})\}$ of order $3$ over $\mathbb{Q}$.  (A minimal model for $E$ is $y^2+y=x^3+x^2-9x-15$, and its conductor is $19$.  This is curve 19a1 in Cremona's tables on his elliptic curve website \cite{cr}.)  The assertion about torsion follows by looking at the reduced curve $\widetilde{E}$ of $E$ mod 5, which has only the points $\tilde O, (2,2),(2,3)$ over $\mathbb{F}_5$.  Since the torsion group over $\mathbb{Q}$ injects into the group $\widetilde{E}(\mathbb{F}_5)$, the points in $T$ are the only torsion points.  Hence, $E(\mathbb{Q}) = T$.\medskip

Both of the points $(\frac{16}{3}, \pm \frac{19}{2})$ yield the point $(X,Y) = (0, 1)$ on $F_1(x,y) = 0$, while the infinite point $O$ on $E$ yields $(X,Y) = (-1, 1)$.  Hence, the only points on $F(x,y) = 0$ are $(0,n)$ and $(-n,n)$.  Since both of these points are not allowed, this shows that $F(x,y) \neq 0$.  Equation (\ref{eqn:37}) now implies that $y=n$.  Now we compute
\begin{align*}
t_1(x,n) - t_1(m,n) &= (m - x)(m^2 + 2mn+ n^2 + 2nx + mx  + x^2)\\
t_2(x,n)- t_2(m,n) & = (m - x)(m^2 - n^2 + mx  + x^2);
\end{align*}
and find that
$$\textrm{Res}_x(m^2 + 2mn+ n^2+ + 2nx + mx  + x^2, m^2 - n^2 + mx  + x^2) = 4mn^2(m + n).$$
It follows that the quadratic polynomials in this resultant are not both 0, from which we obtain $x = m$. \medskip

Therefore, $(x,y) = (m,n)$ in this case. \medskip

\noindent {\it Case 2.}  Assume now that $t_1(x,y) = t_1(m,n)$ and $t_2(x,y) = -t_2(m,n)$. We compute that
\begin{align*}
\textrm{Res}_m(t_1(x,y) &- t_1(m,n), t_2(x,y) + t_2(m,n)) \ = 8F_2(x,y),\\
& = 8(y^6n^3 + 2xy^5n^3 + 4x^2y^4n^3 + (-n^6 + 4n^3x^3 - x^6)y^3\\
& \ \ \ + (-n^6x + 6n^3x^4 - 3x^7)y^2 + (-4n^6x^2 + 7n^3x^5 - 3x^8)y\\
& \ \ \  + n^9 - 3x^3n^6 + 3x^6n^3 - x^9).
\end{align*}
In this case, $F_2(x,y)$ is homogeneous of degree $9$ in $(x,y,n)$, and setting $F_3(x,y) = \frac{1}{n^9}F_2(nx,ny)$ yields the polynomial
\begin{align*}
F_3(x,y) = \ & \ y^6+ 2xy^5 + 4x^2y^4 + (-1 + 4x^3 - x^6)y^3 + (-x + 6x^4 - 3x^7)y^2\\
&  + (-4x^2 + 7x^5 - 3x^8)y + 1 - 3x^3 + 3x^6 - x^9.
\end{align*}
The curve $F_3(x,y) = 0$ has genus 10, and has only a finite number of rational points, by the well-known theorem of Faltings \cite{fa}.  Let $(x_0,y_0)$ be a rational point on $F_3(x,y) = 0$ other than $(1,0)$.  Then $F_2(nx_0,ny_0) = 0$.  There are at most two integers $n$ for which $x = nx_0, y=ny_0$ are relatively prime integers, since $n$ must be divisible by the least common multiple $\l$ of the denominators of $x_0$ and $y_0$ and $|n| > \l$ implies $\gcd(x,y)>1$.  Hence, there can be at most two values of $n$ for each point $(x_0,y_0)$. \medskip

Using the same argument applied to the resultant with respect to $n$ yields the curve $F_4(x,y) = 0$, where
\begin{align*}
F_4(x,y) = \ & \ y^9 + 3xy^8 + 6x^2y^7 + (7x^3 - 2)y^6 + (6x^4 - 4x)y^5 + (3x^5 - 6x^2)y^4\\
& + (x^6 - 4x^3 + 2)y^3 + (-2x^4 + 2x)y^2 + x^2y - 1 - x^3.
\end{align*}
This curve also has genus $10$, so there are only a finite number of integers $m$ for which $(x,y) = (mx_1,my_1)$ is a point on $F_2(x,y) = 0$ satisfying $\gcd(x,y) = 1$.  Hence, there are at most finitely many pairs $(m,n)$ for which $t_1(x,y) = t_1(m,n)$ and $t_2(x,y) = -t_2(m,n)$ for some pair of relatively prime integers $(x,y)$.
\medskip

\noindent {\it Cases 3, 4.}
The cases $t_1(x,y) = t_2(m,n), t_2(x,y) = \pm t_1(m,n)$ are handled by the same arguments as in Case 2.  The genera of the curves $F_5 = 0$ and $F_6 = 0$ defined by
\begin{align*}
\textrm{Res}_m(&t_1(x,y) - t_2(m,1), t_2(x,y) - t_1(m,1)) = 8F_5(x,y),\\
F_5(x,y) & = \ x^6 - x^3 + 1 + (x^5 + 2x^2)y + (2x^4 + 3x)y^2 + (x^6 + 4x^3 + 3)y^3\\
& \ \ + (3x^5 + 8x^2)y^4 + (6x^4 + 6x)y^5 + (7x^3 + 3)y^6 + 6x^2y^7 + 3xy^8 + y^9
\end{align*}
and
\begin{align*}
\textrm{Res}_m(&t_1(x,y) - t_2(m,1), t_2(x,y) + t_1(m,1)) = -8F_6(x,y),\\
F_6(x,y) & = \ y^6 + 2xy^5 + 2x^2y^4 + (x^6 - 1)y^3 + (3x^7 + 2x^4 - x)y^2\\
& + (3x^8 + 5x^5 + 2x^2)y + x^9 + 3x^6 + 3x^3 + 1
\end{align*}
are, respectively, $7$ and $10$; and these are also the genera of the curves obtained by taking the resultants of $t_1(x,y) - t_2(1,n)$ and $t_2(x,y) \mp t_1(1,n)$ with respect to $n$.  It follows as before that there are at most finitely many pairs $(m,n)$ for which the simultaneous Thue equations in these cases have an integer point $(x,y)$ with $xy(x+y) \neq 0$ and $\gcd(x,y) = 1$.  This proves the theorem.
\end{proof}

Andrew Bremner (private communication) has shown us that the system
$$t_1(x,y) - t_1(m,n) = 0, \ \ t_2(x,y) + t_2(m,n) = 0$$
in Case 2 above can be mapped to the intersection of the surfaces in $\mathbb{P}^3$ whose equations are
\begin{align*}
& Y^2 - (M+3N)Y + N^2  - X^2 = 0,\\
& (M + 2N + X)Y^2 - (M^2 + 5MN + 7N^2 + 2NX)Y\\
& \ \  + 2XMN + 4XN^2 + M^3 + 5M^2N + 9MN^2  + 6N^3 = 0,
\end{align*}
where $X, Y, M, N$ are polynomials in $\mathbb{Z}[x,y,m,n]$.  Taking the resultant of these two polynomials with respect to
$M$ yields the genus $4$ curve
\begin{align*}
\mathscr{C}: & \ Y^6 + (-4N + X)Y^5 + (7N^2 - 2X^2)Y^4\\
 & + (-10N^3 - 2N^2X + 7NX^2)Y^3\\
 & + (8N^4 + 2N^3X - 10N^2X^2 - 2NX^3 + 2X^4)Y^2\\
 & + (-4N^5 + 8N^3X^2 - 4NX^4)Y + N^6 - 3X^2N^4 + 3X^4N^2 - X^6 = 0.
\end{align*}
This curve possesses the points $(X,Y,N) = (1, 0, 1), (-1, 0, 1), (-1, 1, 1)$, but it is unclear if it has any other points defined over $\mathbb{Q}$.

\begin{thm}
i) There are infinitely many vertices in $\Gamma$ at which at least three distinct $c$-triangles meet.  This holds for all vertices of the form
$$a = x^6 - 3x^5y + 5x^4y^2 - 5x^3y^3 + 5x^2y^4 - 3xy^5 + y^6, \ \ x, y \in \mathbb{Z}, (x,y) = 1,$$
for which the integers $x^2, y^2, (x-y)^2$ are distinct. \smallskip

ii) A root $\alpha$ of the normal polynomial
$$h(t)= t^6 - 3t^5 + 5t^4 - 5t^3 + 5t^2 - 3t + 1$$
generates the Hilbert class field $\Sigma = \mathbb{Q}(\alpha)$ of $\mathbb{Q}(\sqrt{-23})$ over $\mathbb{Q}$.  Hence, the values $a = y^6 h\left(\frac{x}{y}\right) = \textsf{N}_{\Sigma/\mathbb{Q}}(x-y\alpha)$ in (i) are norms from the field $\Sigma$. \smallskip 

iii) For any vertex of the form $a = y^6 h\left(\frac{x}{y}\right)$, the prime divisors $p \neq 23$ of $a$ split completely in the field $K = \mathbb{Q}(\gamma)$.
\label{thm:12}
\end{thm}

\begin{proof}
i) We consider the pairs
\begin{align}
\notag (m,n) = (-x^2,& \ x^2-xy+y^2), (-(x-y)^2,x^2-xy+y^2),\\
\label{eqn:39} & \ (-y^2,x^2-xy+y^2),
\end{align}
for integers $x,y$ with $(x,y) = 1$ and for which $x^2, y^2, (x-y)^2$ are distinct.  In particular, $x, y, x-y \neq 0$.  We have
\begin{equation}
\label{eqn:40} t_1(m,n) = x^6 - 3x^5y + 5x^4y^2 - 5x^3y^3 + 5x^2y^4 - 3xy^5 + y^6 = y^6h\left(\frac{x}{y}\right)
\end{equation}
for all three of these pairs.  The values of $c$ corresponding to these pairs have the respective numerators
\begin{align*}
A (-x^2,&x^2-xy+y^2) = (x^6 - x^5y + x^4y^2 - x^3y^3 + x^2y^4 - xy^5 + y^6)\\
& \ \times (x^6 - 3x^5y + 9x^4y^2 - 13x^3y^3 + 11x^2y^4 - 5xy^5 + y^6),\\
A (-(x-y)^2,&x^2-xy+y^2) =\\
& \ \ (x^6 - 3x^5y + 9x^4y^2 - 13x^3y^3 + 11x^2y^4 - 5xy^5 + y^6)\\
& \ \times (x^6 - 5x^5y + 11x^4y^2 - 13x^3y^3 + 9x^2y^4 - 3xy^5 + y^6),\\
A(-y^2,&x^2-xy+y^2) = (x^6 - x^5y + x^4y^2 - x^3y^3 + x^2y^4 - xy^5 + y^6)\\
& \ \times (x^6 - 5x^5y + 11x^4y^2 - 13x^3y^3 + 9x^2y^4 - 3xy^5 + y^6).
\end{align*}
This is the same pattern displayed by the values of $c$ in (\ref{eqn:23}).  Since the three $(m,n)$ pairs in (\ref{eqn:39}) are allowable solutions of (\ref{eqn:40}), Theorem \ref{thm:10} implies the assertion. \medskip

ii) The polynomial $h(t) =  t^6 - 3t^5 + 5t^4 - 5t^3 + 5t^2 - 3t + 1$ is a normal polynomial with discriminant $-23^3$.  This follows from the fact that $\textrm{Gal}(h(t)/\mathbb{Q}) = D_3$ (the anharmonic group) is generated by the automorphisms
\begin{align*}
\sigma(\alpha) & = \frac{\alpha-1}{\alpha} = \alpha^5 - 3\alpha^4 + 5\alpha^3 - 5\alpha^2 + 5\alpha - 2, \ \ \sigma^3 = 1,\\
\tau(\alpha) & = 1- \alpha, \ \ \tau^2 = 1;
\end{align*}
where $\alpha$ is a root of $h(t)$ and $\tau \sigma = \sigma^2 \tau$.  The factorization 
\begin{align*}
h(t) & = \left(t^3 + \frac{-3+\sqrt{-23}}{2} t^2 + \frac{-3-\sqrt{-23}}{2} t + 1\right)\\
& \ \times \left(t^3 + \frac{-3-\sqrt{-23}}{2} t^2 + \frac{-3+\sqrt{-23}}{2} t + 1\right)
\end{align*}
shows that the splitting field contains $\mathbb{Q}(\sqrt{-23})$.  Finally, the discriminant of each of these factors is $1$, so that the splitting field is unramified over $\mathbb{Q}(\sqrt{-23})$.  It follows that this splitting field is the Hilbert class field $\Sigma$ of $\mathbb{Q}(\sqrt{-23})$, since $[\Sigma:\mathbb{Q}(\sqrt{-23})] = 3$.  This implies the remaining assertion in (ii).\medskip

iii) As in the last paragraph of the proof of Theorem \ref{thm:3}, if $p$ is a prime dividing $a = y^6 h\left(\frac{x}{y}\right)$, with $(x,y)=1$, then $p \nmid y$ and $p$ is a prime divisor of the normal polynomial $h(t)$.  By \cite[Thm. 4]{gb}, any normal polynomial splits completely modulo $p$ for all but finitely many of its prime divisors.  However, since the discriminant of $h(t)$ is equal to the field discriminant $d(\Sigma/\mathbb{Q}) = -23^3$, the powers $\{1, \alpha, \dots, \alpha^5\}$ form an integral basis for the ring of integers in $\Sigma$.  In this case the proof of \cite[Thm. 4]{gb} shows that $h(t)$ factors into linear factors mod $p$ for all of its prime divisors.  If $p \neq 23$, these linear factors are distinct, so that $p$ is unramified and splits completely in the field $\Sigma$, implying that it splits completely in the subfield $K \subset \Sigma$.
\end{proof}

Theorem \ref{thm:12} shows that Conjecture \ref{conj:2}(iv) holds for infinitely many vertices. \medskip

\noindent {\bf Remarks.} By virtue of $y^6h\left(\frac{x}{y}\right) = (x^2-xy+y^2)^3-x^2y^2(x-y)^2$, each of the vertices in Theorem \ref{thm:12}(i) is a cube minus a square.  In terms of a root $\alpha$ of $h(x)$, the roots of $x^3-x-1$ are
\begin{align*}
\gamma_1 & = (\alpha^2+1)(\alpha^2-2\alpha+2) = (\alpha^2+1) \cdot \tau(\alpha^2+1),\\
\gamma_2 & = -\alpha (\alpha^2+1)(\alpha^2-2\alpha+2) = \sigma(\gamma_1),\\
\gamma_3 & = (\alpha-1)(\alpha^2+1)(\alpha^2-2\alpha+2) = \sigma^2(\gamma_1) = \tau(\gamma_2).
 \end{align*}
 
 Using Theorem \ref{thm:12} we have found six more vertices where $5$ triangles touch.  These vertices, together with the solutions $(m,n)$ of $t_1(m,n) = a$, are:
 \begin{align*}
 a = 2019658087, \ \ (m,n) = &(-21235, 12103), (-1369, 1267), (-1156, 1267),\\
 & (-9, 1267), (1458, -275);\\
 a = 4659789889, \ \ (m,n) = &(-8431, 4840), (-2209, 1729), (-1024, 1729),\\
 &  (-225, 1729), (1897, -324);\\
 a = 27115751629, \ \ (m,n) = &(-3969, 3109), (-1849, 3109), (-400, 3109),\\
 & (6437, -3304), (18953, -10759);\\
 a = 295789896739, \ \ (m,n) = &(-7225, 6679), (-6084,6679), (-3591, 6955),\\
 & (-2722, 7021), (-49, 6679);\\
 a = 823905321247, \ \ (m,n) = &(-40249,23212), (-12769, 9787), (-5041, 9787),\\
 & (-1764, 9787), (1934, 8497);\\
 a = 285605862810841, \ \ (m,n) = & (-162689, 98505), (-81225, 67081), (-48841, 67081),\\
 & (-12817, 68809), (-4096, 67081).
 \end{align*}
These integers are the values $a = y^6h(x/y)$ for $(x,y) = (37,3), (47,15), (63,20), (85,7), (113, 42), (285, 64)$, respectively.  These $x$-values are values of the polynomial $3k^2+7k+37$ for $0 \le k \le 4, k = 8$.

\section{Rational periodic cycles of the maps $f_c(x)$.}

Since the map $f_{-29/16}(x) = x^2 - \frac{29}{16}$ is a main focus of this paper, we prove the following theorem.  For $c = -\frac{29}{16}$ this result was stated in \cite[p. 18]{p}, with remarks on how to verify it using a finite but non-explicit calculation.  Here we use a simpler, arithmetic approach based on \cite{mp} and \cite{ms}.  See also \cite{pe} and \cite[ch. 2]{z}.  The idea of the proof is that a rational $n$-cycle of $f_c(x)$ lies in the field $\mathbb{Q}_p$ of $p$-adic numbers, for any prime $p$.  The assumption that $p$ is a prime of good reduction for $f_c(x)$, i.e., that $p$ does not divide the denominator of $c$, leads to a restriction on the possible period $n$ of any $p$-adic cycle.  Comparing the possible periods for different primes (under suitable hypotheses) shows that these restrictions are incompatible unless $n = 3$.

\begin{thm}
If $c = -\frac{A(m,n)}{B(m,n)}$, where $29 \mid A(m,n)$ and either $7 \nmid B(m,n)$ or $11 \nmid B(m,n)$, then the only rational periodic cycle of the map $f_c(x)$ is the rational $3$-cycle.
\label{thm:13}
\end{thm}

\begin{proof} For the proof of this theorem and the next we write $A = A(m,n), B = B(m,n)$, to free the letters $m,n$ for use as periods of maps.  Let $\textsf{P}$ denote the ring of rational numbers whose denominators are divisible at most by primes which divide $B$.  $\textsf{P} = \mathbb{Z}_S$ is the localization of $\mathbb{Z}$ at the submonoid $S = \langle \pm p: p \mid B \rangle$ of $\mathbb{Z}$.  (See \cite[pp. 393-395]{j}.) Then $\textsf{P}$ has unique factorization, its units make up the group $\textsf{P}^\times = \langle \pm p: p \mid B \rangle \le \mathbb{Q}^\times$, and its primes are associates of the primes which do not divide $B$.  Moreover, for these primes $\textsf{P}/(p) \cong \mathbb{Z}/(p)$.  The monic polynomials
$$F_n(x) = f^n(x) - x,  \ \ n \ge 1,$$
have coefficients in $\textsf{P}$, and therefore have roots which are integral over $\textsf{P}$.  If $\alpha \in \mathbb{Q}$ is a periodic point of minimal period $n$, it follows that $\alpha \in \textsf{P}$.  From \cite{mp} we have that
$$F_n(x) = \prod_{d \mid n}{\Phi_d(x)}, \ \ \Phi_d(x) = \prod_{k \mid d}{(f^k(x)-x)^{\mu(d/k)}},$$
where the $\Phi_d(x)$ are polynomials in $\textsf{P}[x]$, among whose roots are all the periodic points of $f(x)$ of primitive period $d$.  Thus $\alpha$ is a root of $\Phi_n(x)$. \medskip

We will use a characterization of the minimal period $n$ from \cite[Thm. 1.1]{ms}, which is implicit in the proofs of \cite[Lemmas 2.2, 2.3]{mp}.  This result implies that if $p \nmid B$, the minimal period $n$ of a periodic point $\alpha$ of $f$ in the $p$-adic field $\mathbb{Q}_p$ satisfies $n = m, mr$, or $mrp^e$, where $m \le p, r \mid p-1$ or $r = \infty$, and $e \ge 1$.  In the notation of \cite{mp} and for the case we are considering, this says the following over the finite field $\textsf{P}/(p) = \mathbb{F}_p$.  Let $m$ be the minimal period of $\tilde \alpha \equiv \alpha$ (mod $p$) under the action of $\tilde f(x) \equiv f(x)$ (mod $p$); assume that $p \nmid m$. \medskip

1. If $\tilde \alpha$ is a multiple root of $\Phi_m(x)$ (mod $p$), then the multiplier of the orbit containing $\alpha$ satisfies
$$\mu = (f^m)'(\alpha) \equiv 1 \ (\textrm{mod} \ p),$$
and $\tilde \alpha$ cannot be a root of any other $\Phi_k(x)$ (mod $p$) with $p \nmid k$.  In this case $r=1$ and $n = m$ or $mp^e$. \smallskip

2. If $\tilde \alpha$ is a simple root of $\Phi_m(x)$ (mod $p$) and the multiplier of the orbit satisfies
$$\mu = (f^m)'(\alpha) \not \equiv 0 \ (\textrm{mod} \ p),$$
then $r$ is the order of $\mu$ (mod $p$) and $\tilde \alpha$ is a multiple root of $\Phi_{mr}(x)$ (mod $p$). In this case $n = m, mr$ or $mrp^e$.  \smallskip

3. If $\mu \equiv 0$ (mod $p$), then $\tilde \alpha$ cannot be a root of any other polynomial $\Phi_k(x)$ (mod $p$), so $n = m$.  (Here $r$ is taken to be $\infty$.  For this case see (2.2) and the first paragraph of Case B in the proof of \cite[Lemma 2.3]{mp}.)  \medskip

Finally, note that $p \mid m$ is impossible for the quadratic map $f(x)$, since this would imply that $m = p$ and $f(x)$ is $1-1$ on $\mathbb{F}_p$, which it is not.  In fact, it is easy to see that $m \le \frac{p+1}{2}$.  This is because the elements $a_i$ of an $m$-cycle in $\mathbb{Z}/(p)$ satisfy $a_{i-1}^2 \equiv a_i - c$, so that $a_i-c$ is either $0$ or a quadratic residue (mod $p$).  Since the $m$ elements $a_i - c$ are distinct, this gives $m \le \frac{p+1}{2}$. \medskip

Reducing modulo $29$ gives that $f_c(x) \equiv \tilde f(x) = x^2$.  Computing the action of $\tilde f(x)$ on $\mathbb{Z}/(29)$ gives the following cycles and their corresponding multipliers:
$$\{0\},\mu = 0;\ \ \{1\}, \mu = 2;\ \ \{16,24,25\}, \mu= 8; \ \  \{7,20,23\}, \mu = 8.$$
Thus, we have $r = \infty$ in the first case and $r = 28$ for the other three cycles, since $2$ is a primitive root (mod $29$).  Hence, the period $n$ of a $29$-adic periodic cycle must be:
\begin{align*}
&n = 1 \ \textrm{for} \ \{0\};\\
&n = 1, 28 \ \textrm{or} \ 28 \cdot 29^e \ \textrm{for} \ \{1\};\\
&n = 3, 84 \ \textrm{or} \ 84 \cdot 29^e \ \textrm{for} \ \{16,24,25\} \ \textrm{or} \  \{7,20,23\}.
\end{align*}
Hence, we have $n=1$ or $3$ or $n \equiv 0$ (mod $28$) for a $29$-adic cycle.  Note that $f_c(x)$ can have a $29$-adic fixed point: for $c = -29/16$, the fixed points of $f_c(x)$ are $\frac{1}{2}\pm \frac{\sqrt{33}}{4}$, both of which lie in the $29$-adic field $\mathbb{Q}_{29}$. \medskip

If $7 \nmid B$, we consider the possible periodic points of $f_c(x)$ in $\mathbb{Q}_7$.  By Theorem \ref{thm:3}, $A \equiv 0$ or $1$ mod $7$, and $B = 16C^2$ is a square mod $7$, so we have the possibilities $c \equiv 0, -1, -2, -4$ (mod $7$).  The map $\tilde f(x) = x^2-1$ does not occur in our situation, since
\begin{equation*}
\Phi_{3,\tilde f}(x) \equiv (x^3 + 5x^2 + 3x + 2)(x^3 + 3x^2 + x + 4) \ (\textrm{mod} \ 7),
\end{equation*}
where the cubics are irreducible (mod $7$), so that $\tilde f$ does not have a rational $3$-cycle.  The cycles (mod $7$) of the other three maps are as follows.
\begin{align*}
\tilde f(x) = x^2: & \ \{0\}, (\mu, r) = (0, \infty), \ n = 1;\\
& \ \{1\}, (\mu, r) = (2, 3), \ n = 1, 3, 3 \cdot 7^e;\\
& \ \{2,4\}, (\mu, r) = (4, 3), \ n = 2, 6, 6 \cdot 7^e;\\
\tilde f(x) = x^2-2: & \ \{2\}, (\mu,r) = (4, 3), \ n = 1, 3, 3 \cdot 7^e;\\
& \ \{6\}, (\mu,r) = (5, 6), \ n = 1, 6, 6 \cdot 7^e;\\
\tilde f(x) = x^2-4: & \ \{0, 3, 5\}, (\mu,r) = (0, \infty), \ n = 3.
\end{align*}
This shows that $n = 1$ or $3$ or $n \not \equiv 0$ (mod $4$) for a $7$-adic cycle. \medskip

Putting the information together from the primes $7, 29$ gives that the period of a rational periodic point of $f_c(x)$ must be $n = 1$ or $3$.  But we know $f_c(x)$ has a rational $3$-cycle, so it cannot have a rational fixed point, by Poonen's result \cite[Thm. 2]{p}.  Hence, the rational $3$-cycle is the only rational cycle of the map $f_c(x)$.  Note that $f_c(x)$ cannot have more than one rational $3$-cycle, by \cite[Thm. 3]{m}.  \medskip

Alternatively, if $11 \nmid B$, we can consider the $11$-adic periodic points of $f_c$.  We find the following possibilities:
\begin{align*}
\tilde f(x) = x^2+3: & \ \{6\}, (\mu, r) = (1, 1), \ n = 1, 11^e;\\
& \ \{1,4,8\}, (\mu, r) = (3, 5), \ n = 3, 15, 15 \cdot 11^e;\\
\tilde f(x) = x^2+7: & \ \{5, 10, 8\}, (\mu,r) = (10, 2), \ n = 3, 6, 6 \cdot 11^e;\\
\tilde f(x) = x^2+8: & \ \{1,9\}, (\mu,r) = (3, 5), \ n = 2, 10, 10 \cdot 11^e;\\
& \ \{0, 8, 6\}, (\mu,r) = (0, \infty), \ n = 3.
\end{align*}
Thus, we also get that $n = 1$ or $3$ or $n \not \equiv 0$ (mod $4$) for an $11$-adic cycle, yielding the only possibility $n=3$ when we put this together with the data for $29$ and Poonen's theorem.  This proves the theorem.
\end{proof}

\begin{cor}
Besides the rational periodic points $-\frac{1}{4}, -\frac{7}{4}, \frac{5}{4}$ of period $3$, the polynomial map $f(x) = x^2-\frac{29}{16}$ has no other rational periodic points.  In particular, $f(x)$ has exactly $8$ rational periodic or preperiodic points.
\label{cor:4}
\end{cor}

The second assertion in this corollary was proved in \cite{p}. \medskip

\noindent {\bf Remarks.}  1. If $f(x)$ is extended to the rational map $f(X,Y) = (16X^2-29Y^2,16Y^2)$ on the projective line $\mathbb{P}^1(\mathbb{Q})$, then $f$ has $9$ rational periodic and preperiodic points, since $\infty=(1,0) \in \mathbb{P}^1(\mathbb{Q})$ is a fixed point. \medskip

\noindent 2. Since we have $2$-adically that
$$-\frac{29}{16} = -\frac{1}{16} -\frac{3}{4}-1 = -\frac{1}{16} -\frac{3}{4}+\sum_{i=1}^\infty{\frac{3}{2}(-2)^i},$$
it follows by setting $q = -2$ in Proposition 5 of \cite[pp. 96-97]{m1} that $f(x) = x^2-\frac{29}{16}$ has $2$-adic periodic points of all possible periods $n \ge 1$.  In fact, the polynomial $\Phi_{n,f}(x)$ splits completely in the $2$-adic field $\mathbb{Q}_2$ (and has distinct roots), for all $n \ge 1$. \smallskip

The following result uses the same method of proof, but requires a stronger result of Pezda \cite[Thm. 2]{pe} (see also Zieve \cite[p. 12]{z}).  This result says that if a polynomial $f(x)$ has $p$-adic integral coefficients, so that $f(x)$ has good reduction at $p$, then the period of a $p$-adic periodic point must satisfy $n = m$ or $n= mr$ if $p > 3$; and $n = m, mr$ or $mrp$, if $p = 2$ or $3$.  For a quadratic map $m$ and $r$ must satisfy $m \le \frac{p+1}{2}$ and $r \mid p-1$.

\begin{thm} With $c = -\frac{A(m,n)}{B(m,n)}$, the only rational cycle of of $f_c(x) = x^2 + c$ is the rational $3$-cycle, if one of the following conditions holds:
(i) $3 \nmid B(m,n)$; (ii) $5 \nmid B(m,n)$; or (iii) $(7 \cdot 29, B(m,n)) = 1$.
\label{thm:14}
\end{thm}

\begin{proof}
(a)  If $3 \nmid B$, then the factorizations
\begin{align*}
\Phi_{3,x^2}(x) & \equiv x^6 + x^5 + x^4 + x^3 + x^2 + x + 1 \ (\textrm{mod} \ 3),\\
\Phi_{3,x^2+2}(x) & \equiv (x^3 + 2x^2 + 1)^2 \ (\textrm{mod} \ 3),
\end{align*}
show that $c \equiv 1$ (mod $3$).  The only cycle of the map $\tilde f(x) = x^2 + 1$ (mod $3$) is the $1$-cycle $\{2\}$, with multiplier $\mu = 1$.  Hence a $3$-adic periodic point of $f_c(x)$ satisfies $n = m = 1$ or $n = mrp = 3$, by Pezda's theorem.  \medskip

(b) The only value of $\tilde c$ (mod $5$), for which $\tilde f(x) = x^2 + \tilde c$ has a $3$-cycle in $\mathbb{Z}/(5)$, is $\tilde c \equiv 1$ (mod $5$).  This can be seen from the following factorizations modulo $5$:
\begin{align*}
\Phi_{3,x^2}(x) & \equiv x^6 + x^5 + x^4 + x^3 + x^2 + x + 1,\\
\Phi_{3,x^2+2}(x) & \equiv (x^3 + 3x^2 + 4x + 3)^2,\\
\Phi_{3,x^2+3}(x) & \equiv (x^3 + x^2 + 3x + 4)(x^3 + 2x + 1),\\
\Phi_{3,x^2+4}(x) & \equiv x^6 + x^5 + 3x^4 + 4x^3 + x^2 + 1.
\end{align*}
Hence $c \equiv 1$ (mod $5$).  Here the only cycle is $\{0,1,2\}$ and
$$\Phi_{3,x^2+1}(x) \equiv x(x + 4)(x + 3)(x^3 + 4x^2 + 4x + 2) \ (\textrm{mod} \ 5).$$
The multiplier of the cycle $\{0,1,2\}$ is
$$\mu \equiv f'(0) f'(1) f'(2) \equiv 0 \ (\textrm{mod} \ 5),$$
so $r = \infty$ and $n = m = 3$.  Thus, the only $5$-adic cycle of $f$ is the rational $3$-cycle.  \smallskip

(c) The only values of $c \not \equiv 0$ (mod $29$), for which $\tilde f(x) \equiv f_c(x)$ has a $3$-cycle in $\mathbb{Z}/(29)$, are:
$$c \equiv 11, 14, 15, 20, 21, 27 \ (\textrm{mod} \ 29).$$
This may be verified by factoring $\Phi_{3,\tilde f}(x)$ for each $c$ (mod $29$), or by checking that these are the only values of $y_1(s)$ mod $29$, where $s \not \equiv 0, -1$ (mod $29$).  The cycles and corresponding data for each of these $c$-values (mod $29$) are listed in Table \ref{tab:2}.  In the last column the values of $m$ and $mr$ are listed.  All the $29$-adic cycles with $4 \mid n$ and $n = 14$ can be excluded using the $7$-adic values for $n$ from the proof of Theorem \ref{thm:13}.  This leaves $n = 1, 2, 3, 21, 42$, of which $n = 1, 2$ are not periods of rational cycles, by \cite{p}.  However, the theorem of Pezda shows that $7$ does not divide the period of a $7$-adic cycle of a quadratic polynomial map, since $m \le 4$ and $r \mid 6$, and this rules out the possibilities $n = 21, 42$.
\end{proof}

\begin{table}
  \centering 
  \caption{Cycles for $\tilde f(x) = x^2 +c$ mod $29$.}\label{tab:2}

\noindent \begin{tabular}{|c|clc|c|c|}
\hline
& & & &\\
$c$	&   cycle & $\mu$ & $r$ & $n$  \\
\hline
11 & \{4, 27, 15\} & 26 & 28 & 3, 84\\
 & \{6, 18, 16\} & 20 & 7 & 3, 21\\
 14 & \{8, 20\} & 2 & 28 & 2, 56\\
  & \{18, 19, 27\} & 9 & 14 & 3, 42\\
15 & \{9\} & 18 & 28 & 1, 28\\
& \{21\} & 13 & 14 & 1, 14\\
& \{6, 22\} & 6 & 14 & 2, 28\\
& \{10, 28, 16\} & 25 & 7 & 3, 21\\
20 & \{13, 15\} & 26 & 28 & 2, 28\\
& \{7, 11, 25\} & 1 & 1 & 3\\
21 & \{14\} & 28 & 2 & 1, 2\\
& \{16\} & 3 & 28 & 1, 28\\
& \{8, 27, 25\} & 19 & 28 & 3, 84\\
27 & \{2\} & 4 & 14 & 1, 14\\
& \{28\} & 27 & 28 & 1, 28\\
& \{5, 23\} & 25 & 7 & 2, 14\\
& \{3, 7, 18\} & 8 & 28 & 3, 84\\
& \{4, 14, 20, 21\} & 16 & 7 & 4, 28\\
 &  &  &  & \\
  \hline
\end{tabular}
\end{table}

The above arguments suggest the following theorem.

\begin{thm} Let $p$ be an odd prime.  The number $N(p)$ of residue classes $c \in \mathbb{Z}/(p)$, for which the polynomial $\Phi_{3,\tilde f}(x)$ for the map $\tilde f(x) =x^2 + c$ has a linear factor (mod $p$), satisfies
$$N(p) \le \frac{1}{3}\left(p + 2\left(\frac{-3}{p}\right)\right).$$
\label{thm:15}
\end{thm}

\begin{proof}
Since the parametrization (\ref{eqn:8}), (\ref{eqn:9}) of (\ref{eqn:1}) is valid over $\mathbb{Z}/(p)$, the values of $c$ in the assertion are the values for which $y_1(s) \equiv c$ (mod $p$), for some $s \in \mathbb{Z}/(p) - \{0, -1\}$.  Each such $c$ satisfies
$$c \equiv y_1(s) \equiv y_1(\psi(s)) \equiv y_1(\psi^2(s)) \ (\textrm{mod} \ p),$$
and therefore arises from three values of $s$, unless $\psi(s) = s$ has $s$ as a fixed point (mod $p$).  The latter situation occurs if and only if $s^2+s+1 \equiv 0$ (mod $p$) and $x_1(s) \equiv x_2(s) \equiv x_i(s)$ (mod $p$), since
\begin{align*}
x_1(s) - x_2(s) = \frac{s^2+s+1}{s(s+1)}, \ & \ x_1(s) - x_3(s) = \frac{s^2+s+1}{s},\\
x_2(s) - x_3(s) &= \frac{s^2+s+1}{s+1}.
\end{align*}
(Compare with (\ref{eqn:11}) and (\ref{eqn:12}).) Hence, this is the case if and only if $\Phi_{3,\tilde f}(x)$ has a linear factor of multiplicity $3$, in which case the multiplier of the fixed point $x_1(s)$ is
$$\tilde \mu \equiv 2x_1(s) \equiv \frac{s^3+s^2}{s(s+1)} \equiv s \ (\textrm{mod} \ p).$$
If $p \equiv 2$ mod $3$ there are no primitive cube roots of unity (mod $p$), so this situation does not occur and
$$N(p) \le \frac{p-2}{3}.$$
We can have strict inequality here in case $\Phi_{3,\tilde f}(x)$ splits into six linear factors, in which case $c$ arises from two distinct orbits of $\psi(s)$.  On the other hand, if $p \equiv 1$ mod $3$, then there are two primitive cube roots of unity (mod $p$) and 
$$N(p) \le \frac{p-4}{3} + 2 = \frac{p+2}{3}.$$
Since we also have
$$N(3) = 1,$$
the theorem is proved.
\end{proof}

It would be of interest to determine when $\Phi_{3,\tilde f}(x)$ can split completely (mod $p$).  \medskip

\noindent {\bf Acknowledgements.} We are grateful to Joe Silverman, Fangmin Zhou, Andrew Bremner and Roland Roeder for their comments and suggestions.  We also thank Roland Roeder for Figure 1.

\medskip

\noindent Dept. of Mathematical Sciences, LD 270

\noindent Indiana University - Purdue University at Indianapolis (IUPUI)

\noindent Indianapolis, IN 46202

\noindent {\tt e-mail: pmorton@iupui.edu}

\medskip

\medskip

\noindent Mathematics Department

\noindent California State University, Dominguez Hills

\noindent 1000 E Victoria St

\noindent Carson, CA 90747

\noindent {\tt e-mail: sraianu@csudh.edu}

\noindent

\end{document}